\documentclass[10pt]{amsart}

\usepackage{latexsym,amsmath,amssymb,amscd}

\newcommand{\bfi}{\bfseries\itshape}

\newtheorem{theorem}{Theorem}[section]
\newtheorem{lemma}[theorem]{Lemma}
\newtheorem{proposition}[theorem]{Proposition}
\newtheorem{corollary}[theorem]{Corollary}

\theoremstyle{definition}
\newtheorem{definition}[theorem]{Definition}
\theoremstyle{remark}
\newtheorem{remark}[theorem]{Remark}

\numberwithin{equation}{section}

\def\C{\mathbb C}
\def\R{\mathbb R}
\def\S{{\mathcal S} }
\def\g{\mathfrak g}

\def\x{\mathfrak x}
\def\R{\mathbb R}
\def\I{\mathbb I}
\def\N{\mathbb N}

\def\de{\delta}
\def\rh{\rho}

\def\va{\varepsilon}

\def\om{\omega}
\def\va{\varphi}

\def\ol#1{\overline{#1}}

\def\res#1{_{\vert #1}}
\def\inv{^{-1}}

\def\ker#1{\opn{ker}(#1)}
\def\L#1#2{L^{#1}(#2)}

\def\p{\mathfrak{p}}
\def\n{\mathfrak{n}}

\def\om{\omega}
\def\ep{\epsilon}

\def\ch{\chi}
\def\et{\eta}
\def\ps{\psi}

\def\H{{\mathcal H}}
\def\K{{\mathcal K}}

\def\NN{{\mathcal N}}
\def\RR{{\mathcal R}}

\def\X{{\mathcal X}}
\def\U{{\mathcal U}}
\def\V{{\mathcal V}}
\def\PP{{\mathcal P}}

\def\O{{\mathcal O}}
\def\iy{\infty}
\def\no#1{\Vert #1\Vert }

\def\wh{\widehat}

\def\ind{\opn{ind}}

\def\CC{\mathbb C}

\def\RR{{\mathbb  R}}
\def\BB{{\mathbb  B}}

\def\NN{{\mathbb N}}
\def\QQ{{\mathbb Q}}

\def\np{N(\pi)}
\def\Hc{\mathcal H}
\def\Oc{\mathcal O}
\def\Cc{\mathcal C} 
\def\Pc{\mathcal P}
\def\Sc{\mathcal S}

\def\Xc{\mathcal X}
\def\Fc{\mathcal F}
\def\Kc{\mathcal K}
\def\Vc{\mathcal V}
\def\II{\mathcal I}

\def\ep{\varepsilon}
\def\dd {\,\text{\rm d}}

\newcommand{\scalar}[2]{( #1\mid #2)}
\newcommand{\norm}[1]{\Vert #1 \Vert }
\newcommand{\dual}[2]{\langle #1, #2\rangle}

\newcommand{\opn}{\operatorname}

\newcommand{\tr}{\opn{tr}}

\newcommand{\Sg}{\mathfrak{S}}
\newcommand{\zg}{\mathfrak{z}} 
\renewcommand{\gg}{\mathfrak{g}}
\newcommand{\hg}{\mathfrak{h}}

\newcommand{\Xit}{\text{\bfi{X}}}
\newcommand{\Ad}{\opn{Ad}}

\begin{document}

\title[Spectral synthesis]{Spectral synthesis for coadjoint orbits of nilpotent Lie groups}
\author{Ingrid Belti\c t\u a}
\address{Institute of Mathematics ``Simion Stoilow'' 
of the Romanian Academy,  P.O. Box 1-764, Bucharest, Romania}
\email{Ingrid.Beltita@imar.ro}
\author{Jean Ludwig}
\address{Universit\'e de Lorraine, Institut \'Elie Cartan de Lorraine, UMR 7502, Metz, F-57045, France}
\email{jean.ludwig@univ-lorraine.fr}

\date{\today}
\thanks{The first author has been partially supported by  the Grant
of the Romanian National Authority for Scientific Research, CNCS-UEFISCDI,
project number PN-II-ID-PCE-2011-3-0131. }

\begin{abstract}
We determine the space of primary ideals in the group algebra $\L1G $  of  a 
connected nilpotent Lie group by identifying for every $\pi\in\wh G $ 
the family $\II^\pi $ of primary ideals with hull $\{\pi\} $
 with the family of invariant polynomials of a certain finite dimensional subspace $\PP_Q^\pi $  of the space of polynomials $\PP(G) $ on $G $. 
\end{abstract}

\maketitle


\section{Introduction}\label{introd}
Let $G$  be a connected and simply connected nilpotent Lie group and let $\gg$ be its Lie algebra. 
Let $\pi\in \widehat{G}$ be an irreducible unitary representation of $G $. 
Then $\pi $ defines an irreducible unitary representation of the convolution algebra $L^1(G)$, the space  of the measurable functions $f\colon G\to \C $ which are integrable with respect to Haar measure.

Every equivalence class of representations $\pi\in\widehat{G} $ defines the primitive ideal $\ker\pi $ of $\L1G $, and the mapping $\wh G\to Prim(G), \pi\mapsto\ker\pi $ is  a bijection, since $G $ is type I and $^* $-regular. 
Furthermore, the fact that  $G $ is connected and has polynomial growth implies that there exists for every 
$\pi\in\wh G $ a unique minimal ideal $j(\pi) $ with hull $\{\pi\} $ which is contained in every closed 
twosided ideal $I $ with $\opn{hull}(I)=\{\pi\} $ (see \cite{Lu80}). 
It well known that the Schwartz space $\S(G) $ has a dense intersection with $\ker\pi $ (see \cite{Lu83a}).  
This  implies then  that $\ker\pi^N $ is contained in $j(\pi) $ for $N\in\N $ large enough and is dense in it (see \cite{Lu83b}). 

On the other hand Kirillov's orbit picture of the spectrum $ \wh G$ of $G $ tells us that every irreducible unitary representation $\pi $ of $G $ is associated with a coadjoint orbit $\O_\pi\subset \g^* $.
Since the data  $\ker\pi $ and $j(\pi) $ are determined by $\pi $ and hence by the Kirillov orbit 
$\O_\pi  $, one can ask if the geometric structure of the orbit $\O_\pi $ gives us some information on the structure of the algebra $\ker\pi/j(\pi) $. 
For instance, one would like to determine the set $\II^\pi $ of all ideals $I $ in $\L1G $ with hull $\{\pi\} $, the so called primary ideals of $\L1G $. 
In the case where the orbit $\O_\pi $ is flat,  then $\ker\pi/j(\pi)=\{0\} $, i.e., $\pi $ is  a set of spectral synthesis. 
If the group $G $ is step 3, then the structure of the algebra $\ker\pi/j(\pi) $ has been determined in \cite{Lu83a}: The set $\II^\pi $ of all twosided closed ideals of $\L1G $ contained in $\ker\pi $ and containing $j(\pi) $ is in bijection with the set of translation invariant subspaces of a certain finite dimensional translation invariant space $\PP^\pi $ of polynomials on $\g $. 
This space $\PP^\pi $ is determined by a weight type condition coming  from the orbit $\O_{\pi} $.

In this paper we discover for every nilpotent Lie group G and every $\pi\in\wh G $ a finite dimensional translation invariant subspace $\PP_0^\pi=\PP_0 $ of polynomials on $G$, such that the subspace 
$J_{\pi,0}:=\{f\in \S(G)\mid pf\in\ker\pi\text{ for all }p\in\PP_0^\pi\} $ is dense in $j(\pi) $ (Theorem \ref{Jpi0cor}). 
This theorem then implies that there exist for every maximal projection $Q=Q_\xi $ (see Definition \ref{maximalpro}) a subspace $\PP_Q^\pi $  of $\PP_0^\pi $ whose invariant subspaces are in bijection with the set $\II^\pi $ (see Theorem \ref{idetiq}). 
In particular this shows that the Schwartz space $\S(G) $ has a dense intersection with every primary ideal of $\L1G $ (see Theorem \ref{sdense}).
We then study the case where the space $\PP_Q^\pi $ is itself translation invariant (see Theorem \ref{wiseqpqinv}).
Finally we consider the ideals that are $L^\infty(G/N)$ invariant, where $N=\exp\n$ with $\n$  the ideal generated by the
radicals of the elements in the coadjoint orbit corresponding to $\pi$,  and then add some examples. 

The more difficult problem, which is completely open,  is to determine  explicitly for groups of step $\geq 4 $ these spaces $\PP_Q^\pi $  and their relations with the geometric structure of the Kirillov-orbit $\O_{\pi} $. 
For this one needs  precise estimates of the growth of the functions $pc^\pi_{\xi,\et} $, which are products of  polynomial functions $p $ with smooth coefficient functions $c^\pi_{\xi,\et} $. 
This means that the determination of $\PP_Q^\pi $ forces us  
to understand   oscillatory Fourier integrals of the form
\begin{eqnarray*}
 c_{\xi,\eta}(g)=\int\limits_{\R^d}\xi(Q_\pi(g,z))\ol{\eta}(z)^{iP_\pi(g,z)}\, dz, \quad  g\in G,
\end{eqnarray*}
where $\xi$, $\eta\in\S(\R^d) $ and $P_\pi, Q_\pi:G\times\R^d\mapsto \R^d$ are  some special  polynomial mappings coming from the realization of $\pi $ as monomial representation.

\subsection{Notations}\label{notation}

Let as above $\gg $ be  a  nilpotent Lie algebra.
We can realize the group $G=\exp\gg$ as the Lie algebra $\gg$ itself by using the Campbell-Baker-Hausdorff multiplication
$$
 X\cdot Y := X\cdot_{\text{CBH}} Y=X+Y+\frac{1}{2}[X,Y]+\frac{1}{12}[X, [ X,Y]]+\frac{1}{12}[Y, [Y,X]]+\cdots,\ X,Y\in \g.
$$
 The Haar measure  of this group $(\g,\cdot_{\text{CBH}})$ is then Lebesgue measure on the real vector space $\gg $.
 Let $\gg^*$  be the dual of $\gg$
and  denote by 
$$\dual{\cdot}{\cdot}\colon \gg^\ast \times \gg\mapsto \RR $$
the natural duality.

We denote by $\L1G $ the space of  functions $f\colon G\to \C $ which are integrable with respect to Haar measure. 
Then $\L1G $ is an involutive  Banach algebra for the convolution product $\ast$ and the involution ${}^* $
$$
\begin{aligned}
 f\ast g(x) & =\int_G f(y)g(y\inv x)dy\\
 f^*(x)&= \ol{f(x\inv)},  \qquad x\in G.
\end{aligned}
$$
Let also for a function $f:G\to\C $
\begin{eqnarray*}
 \check f(x):=f(x\inv ),\  x\in G.
 \end{eqnarray*}

In this paper $\Sc(\g)$  will denote the Fr\'echet space  of Schwartz functions on the real vector space $\g$, i.e.,  $\Sc(\gg) $  is the space  of rapidly decreasing smooth functions  on $\gg $ and 
$\Sc(G):=\{f\colon G\mapsto \C\mid f\circ \log\in \S(\g)\}$. 
The subspace  $\S(G) $ is then a dense involutive subalgebra of $\L1G $. 

The dual space $\Sc'(G)$ of $\S(G) $ is  the  space of temperate distributions on $G $.
Then $\Sc(G) $ is a subspace of  $BC^\infty(G)$, the space of smooth bounded  functions on $G$, 
with bounded  left and right derivatives,   and $BC^\infty(G)\subset \Sc'(G)$. 

For  a mapping  $\varphi$ from $G $ into a set $X$, denote for $g\in G $ by $\lambda(g)\varphi $ the left translate and by $\rh(g)\varphi $ the right translate of $\varphi $. 

Let $\Pc(G)$ be the space of all polynomial functions on $G$ with complex coefficients, that is, functions $p\colon G \to \CC$ that are polynomials in any polynomial chart of $G$.   
For a $p\in \Pc(G)$ we denote by $\Vc_p$ the subspace of $\Pc(G)$ generated by $G $-left and right translates of $p$, 
\begin{equation}\label{def:Vcp}
 \Vc_p=\opn{span}\{ \lambda(x) \rho(y) p \mid x, y \in G\}.
 \end{equation}
This is a finite dimensional subspace of $\Pc(G)$, and a $G\times G$ submodule of $\Pc(G)$ for the action $\lambda\otimes \rho$.
For $p\in \Pc(G)$ the dimension of $\Vc_p$ is called the degree or $G$-degree of $p$. 

\subsection{Kirillov-theory}\label{kith}
The orbit picture of the spectrum $\wh G $ of $G $, discovered by Kirillov in \cite{Ki62},  
describes the irreducible representations $(\pi,\Hc_\pi) $ of the group $G $ as induced representations. 
For every $\ell\in \gg^* $, there exists a polarization $\p\subset \g $ at $\ell $, i.e.,  $\p $ is a maximal isotropic subspace for the bilinear form $B_\ell(X,Y)=\langle{\ell},{[X,Y]}\rangle$, $X,Y\in \gg $, 
and at the same time a subalgebra of $\gg $. 
To $\p$ and $\ell$ one associates the induced representation $\pi_{\ell,\p}=\opn{ind}_P^G\ch_\ell$, where  
$P=\exp\p $ is the closed connected subgroup  of $G $ with Lie algebra $\p $, and  
$\chi_\ell(\exp X):=e^{-i\langle{\ell},{X}\rangle}$, $X\in \p$ is the unitary character of $P $ whose differential is $-i\ell\res\p $. 
For any  polarization $\p$
at some $\ell\in\g^* $, the representation $\pi_{\ell,\p} $ is  irreducible. 
Finally,  two irreducible representations $\pi_{\ell,\p} $ and $\pi_{\ell',\p'} $ are equivalent if and only if 
$\ell $ and $\ell' $ are contained in the same coadjoint orbit.  
Then for every  irreducible representation $(\pi, \Hc_\pi)$ of $G$ there is an $\ell\in \gg^\ast$ and a polarization $\p$ at $\ell$ such that $\pi$ is equivalent to $\pi_{\ell,\p}$.

Thus there exists a bijection between the space of coadjoint orbits $\g^*/G $ and the spectrum $\wh G $ of $G $ given by the Kirillov mapping
$$
 \Kc \colon  \O\in \g^*/G \mapsto \widehat{G}, \quad \K(\O)=[\pi_{\ell,\p}],
$$
where $[\pi]$ denotes the unitary equivalence class of the representation $\pi$. 
This is actually an homeomorphism (see \cite{Br73}).
 
\subsection{Ideals in $\L1G $}\label{idlone}$ $

For simplicity of notations, an ideal in this paper is always closed and twosided. 
If $I\subset L^1(G) $ is a subspace, not necessarily closed,  which is invariant under left and right multiplication by elements of $L^1(G)$ 
we say that $I $ is an algebra ideal. 

Let $Prim(G) $ be the space of the primitive ideals of the Banach algebra $\L1G $ equipped with the Jacobson topology. Then the mapping 
$$
\wh G\to  Prim(G), \quad  \pi\to \ker \pi \subset {\L1G},  
$$
is a homeomorphism (see \cite{BLSV78}).
Every primitive ideal $I\subset \L1G $ is maximal and every maximal ideal $M\subset\L1G $ is primitive (see \cite{Di60}). 

\begin{definition}\label{hulldef}
\normalfont
i) For an  ideal $I $ of $\L1G $ denote by hull$(I) $ the (closed) subset of $\wh G $ defined by 
$$
 \opn{hull}(I):=\{\pi\in\wh G\mid \pi(I)=\{0\}\}.
$$
\\
ii) For a closed subset $C\subset \wh G $ the kernel of $C$ is the ideal 
 $$
 \ker C:=\{f\in \L1G\mid \pi(f)=0 \text{ forall } \pi\in C\}.
 $$
 \end{definition}

Let $C\subset \wh G $ be a closed subset. 
Then there exists a \textit{minimal (algebra) ideal} $J(C) $ in $\L1G $ with hull $C $. 
This means that there exists a (unique) twosided algebra ideal $J(C) $ in $\L1G $, which has the property that $\pi(J(C))=\{0\} $ if and only if $\pi \in C$ and which is contained in every algebra ideal $I $ of $\L1G $ with $\opn{hull}(I)\subset C $. 
This minimal ideal $J(C) $ is generated by all the self-adjoint Schwartz functions  $f\in\S(G) $  for which there exists a  Schwartz function $g \in\ker C$ such that 
$$
 g\ast f=f\ast g=f. 
$$ 
(see \cite{Lu80}).

Let 
\begin{eqnarray*}
 j(C):=\ol{J(C)}^{\no{}_1}
 \end{eqnarray*}
be the closure in $\L1G $. 
Then $j(C) $ is contained in every ideal  $I\subset \L1G $ with hull$(I)\subset C $.

Let now $C=\{\pi\} $ be a singleton. 
It had been shown in \cite{Lu83b} that for $N\in\N $ large enough the ideal $\ker \pi 	^N $ is contained in $j(\pi) $ and therefore dense in $j(\pi) $. 
Let now
$ N_\pi$ be the smallest such an integer . 
Then
\begin{equation}\label{npidetj}
 \ker \pi ^{N_\pi-1}\not\subset j(\pi), \quad  \ol{\ker\pi^{N_\pi}}^{\no{}_1}=j(\pi).
 \end{equation}

\subsection{Multiplication of convolution products by polynomials}\label{polmulcon}$ $

Let $p\in \Pc(G) $. 
We  choose a Jordan-H\"older basis $\Xc=\{p=p_m,\dots, p_1\}$ of the $G\times G$ submodule $\Vc_p $ of $\PP(G) $ generated by $p $.
This means that $\X $ is a basis of $\Vc_p $ and that 
\begin{equation}\label{trnscoeff}
\begin{aligned}
 \lambda(s)p_j(t) & =\sum_{i=j}^1 a_{i,j}(s)p_i(t),\\ 
\rho(s)p_j(t) & = \sum_{i=j}^1 b_{i,j}(s)p_i(t),
 \end{aligned}
 \end{equation}
where the functions $s\to a_{i,j}(s)$, $b_{i,j}(s)$ are polynomial functions and $a_{j,j}(s)=b_{j,j}(s)=1 $ for all $s\in G $ and $j=1,\dots, m $.

Let $f\in\S(G) $ and $g\colon G\to\CC $ be  a continuous polynomially growing function. 
Then for $x\in G$ one has
\begin{equation*}\label{left}
\begin{aligned}
 p(x)(f\ast g)(x)&= p(x)\int\limits_G f(s)g(s\inv x)ds\\
 &= \int\limits_G p(s s\inv x)f(s)g(s\inv x)ds\\
 &=\sum\limits_{i=1}^m\int\limits_G f(s)p_i(s)b_{i,1}(s\inv x)g(s\inv x)ds
 \end{aligned}
 \end{equation*}
 and similarly
 \begin{equation*}\label{right}
\begin{aligned}
  p(x)(f\ast g)(x)&=p(x)\int\limits_G f(s)g(s\inv x)ds\\
 &=\int\limits_G p(s s\inv x)f(s)g(s\inv x)ds\\
 &=\sum\limits_{i=1}^m\int\limits_G f(s)a_{i,1}(s\inv)p_i(s\inv t)g(s\inv x)ds.
 \end{aligned}
 \end{equation*}

Hence
\begin{equation}\label{cpmulfg}
\begin {aligned}
 p(f\ast g)&=(pf)\ast g+\sum\limits_{i=m-1}^1 (p_i f)\ast(b_{i,m}g)\\
 &= f\ast (p g)+\sum\limits_{i=m-1}^1 (\check a_{i,m} f)\ast(p_ig).
 \end{aligned}
 \end{equation}

\section{The minimal ideal}
\subsection{A class of polynomials given by the growth of coefficients}\label{subsect1}

For a fixed $\ell\in \gg^*$, let $\pi=\pi_\ell\colon G \mapsto \BB(\Hc_\pi)$ be the (unique up to unitary equivalence) unitary 
irreducible representation  corresponding to the coadjoint orbit $\Oc_\ell= \Ad^*(G)\ell$ of $\ell$. 
Then let  
 $\Hc_\pi^\infty$ be the space of smooth vectors for  $\pi$, and $\Hc_\pi^{-\infty}$ its dual. 
We denote by  
$\BB(\Hc_\pi)_\infty$ the space of smooth operators corresponding to $\pi$. 
This is  nothing else than the set  smooth vectors for the irreducible representation 
$\pi \otimes \overline{\pi}   \colon G\times G \mapsto \BB(\Sg_2(\Hc_\pi))$, 
$$ 
(\pi \otimes \overline{\pi}) (g_1, g_2) A  = \pi(g_1) A\pi(g_2^{-1}), \quad g_1, g_2\in G, 
\; A\in  \Sg_2(\Hc_\pi), 
$$
with the topology of $(\Sg_2(\Hc_\pi))^\infty$.
Here $\Sg_2(\Hc_\pi)$ is the space of Hilbert-Schmidt operators on $\Hc_\pi$.

For $A \in \BB(\Hc_\pi)_\infty$ we define the coefficients 
$$c_A(x)=c_{A}^\pi(x) ={\tr(\pi(x)\circ  A)} =\tr(A\circ  \pi(x)),  \quad x\in G. $$
When $A= (\cdot\mid \eta)\xi$, $\xi, \eta\in \Hc_\pi^\infty$ is a projection we simply put 
$c_{A}^\pi= c_{\xi, \eta}^\pi= c_{\xi, \eta}$,  which means
$$
 c_{\xi,\et}^\pi(x)=\scalar{\pi(x)\xi}{\eta}, \quad x\in G.
$$
In particular,  
for $f\in \L1G $ we have  that
$$
\begin{aligned}
 (\check f\ast c_A)(x)&=\int\limits_G f(u\inv){\tr(\pi(u\inv x)\circ A)du}\\
 &={\tr(\pi( x)\circ (A\circ\int_G  f(u\inv )\pi(u\inv) du))}\\
  &={\tr(\pi( x)\circ (A\circ \pi(f)) }\\
  &=c_{ A\circ\pi(f)}(x),  
 \end{aligned}
$$
for every $x\in G $.
Similarly
$$
 c_{A}\ast \check f=c_{\pi(f)\circ A}.
$$
We get thus
\begin{equation}\label{fconco}
 \check f\ast c^\pi_{\xi,\eta}=c_{\xi,(\pi(f)^*)\eta},\quad  c_{\xi,\eta }^\pi\ast \check f=c_{\pi(f)\xi,\eta}. 
 \end{equation}

Consider  an arbitrary fixed
$$\varphi\in (\ker{\pi})^\perp\cap BC^\infty(G).$$
Then, since $\ker\pi$ is  a maximal closed ideal in $L^1(G)$, it follows that 
$$
 \ker \pi=\{f\in L^1(G)\mid \dual{\lambda(x)\rho(y)\varphi}{f} =0,\, \forall x,y\in G\}.
$$

Let 
$V_\varphi \subseteq\Pc(G)$ be the subspace of all polynomial functions $p\in \Pc(G)$ such that
$$
 p(\lambda(x)\rho(y)\varphi)\in L^\infty(G) \text{ for all }x,y\in G.
$$
Then $V_\varphi$ is left and right $G$-invariant, and $\Vc_p\subseteq V_\varphi$ whenever $p\in V_\varphi$.

\begin{theorem}\label{vainv}
The space $V_\varphi$ is contained in $ V_{c_{\eta,\eta'}}$  for every $\eta, \eta'\in\Hc^\infty_\pi$.
Specifically, if $p \in V_\varphi$ and $\{p=p_k, p_{k-1}, \dots, p_1\}$ is a Jordan-H\"older
basis for the finite dimensional subspace $\Vc_p$,  
then there is a constant $C=C(\varphi)$ and $q_{\Hc_\pi^{\infty}}$ a seminorm on $\Hc_\pi^{\infty}$ such that 
$$ \norm{p_j c_{\eta, \eta'}}_{L^\infty(G)} \le C q_{\Hc_\pi^{\infty}}(\eta)q_{\Hc_\pi^{\infty}}(\eta') \, 
\sup\{\norm{p_i \varphi}_{L^\infty(G)}\mid j\le i\le k\}$$
for every $\eta$, $\eta'\in\Hc_\pi^\infty$  and $j=1, \dots, k$.
\end{theorem}
\begin{proof}
Let $\xi,\xi',\eta, \eta'$ be  vectors  in $\Hc_\pi^\infty$. 
There exists a Schwartz function $F_{\xi,\eta}$ on $G$ such that 
$\pi(\check F_{\xi, \eta})=P_{\xi,\eta}= \scalar{\cdot}{\eta}\xi$ 
and for every $N\ge 0$ there exist a seminorm $q_{\Hc_\pi^{\infty}}$  on $\Hc_\pi^{\infty}$ and $C>0$ such that 
$$ \sup_G \vert (1+|s|)^N \check F_{\xi, \eta} (s) \vert \le  q_{\Hc_\pi^{\infty}}(\xi) q_{\Hc_\pi^{\infty}}(\eta).$$
(See \cite{Ho77},  and  also \cite{Pe94}). 
Similarly for other vectors $\eta,\xi',\eta', \xi$.   
Then for any $f\in L^1(G)$ we have that
$$
 \pi(\check{F}_{\xi,\eta}\ast f\ast \check{F}_{\eta',\xi'})=\dual{\check{c}_{\eta,\eta'}}{f}P_{\xi,\xi'}.
$$
Hence, since $\varphi\in (\ker{\pi})^\perp$, 
$$
 \dual{{F_{\xi,\eta}\ast \check\varphi \ast F_{\eta', \xi'}}}{f}=
\dual{{\check{c}_{\eta,\eta'}}}{f} \dual{\check \varphi}{{\check F}_{\xi,\xi'}} \quad \text{for all } f \in L^1(G).
$$
Therefore
\begin{equation}\label{vaniv:1}
\check{F}_{\eta', \xi'}\ast \varphi \ast \check{F}_{\xi,\eta}=\langle{\varphi},F_{\xi,\xi'}\rangle c_{\eta,\eta'}.
\end{equation}

By \eqref{cpmulfg} we have 
\begin{align*} \allowdisplaybreaks
p_j(t)(\check F_{\eta', \xi'}\ast \varphi\ast \check{F}_{\xi, \eta})(t)&=
\int\limits_G \check{F}_{\eta', \xi'}(s)(p_j (\varphi\ast \check F_{\xi', \eta'}))(s^{-1}t)ds\\
&\quad +\sum_{i=j+1}^k\int_G (\check{a}_{i,j}\check F_{\eta', \xi'})(s)( p_i(\varphi\ast
 \check{F}_{\xi, \eta}))(s^{-1} t)ds, 
 \end{align*}
and similarly, 
\begin{align*}
 (p_i (\varphi\ast \check{F}_{\xi, \eta})) (s^{-1} t)
&=\int\limits_G  p_i(y)\varphi(y)\check F_{\xi, \eta}(y^{-1} s^{-1} t)dy\\
&\qquad +\sum_{l=i+1}^k\int\limits_G   p_l(y )\varphi(y)(b_{l, i} \check F_{\xi, \eta})(y^{-1} s^{-1} t))dy.
\end{align*}
Hence there are $N\ge 0$ and $C, C', C''>0$ such that 
$$ 
\begin{aligned}
\vert p_j(t)(\check{F}_{\eta', \xi'}\ast \varphi\ast \check{F}_{\xi, \eta})(t)\vert 
\le C \sup_G\ (1+|s|)^N\vert \check{F}_{\eta', \xi'} (s)\vert \, \sup_{j\le i\le k} 
 \norm{p_j(\varphi\ast \check{F}_{\xi, \eta})}_{L^\infty(G)} \\
\le C' \sup_G (1+|s|)^N \vert \check{F}_{\eta', \xi'} (s)\vert \, \sup_G (1+|s|)^N \vert \check{F}_{\xi, \eta} (s)\vert 
\, \sup_{j\le i\le k} 
 \norm{p_j\varphi}_{L^\infty(G)}\\
 \le C'' q_{\Hc_\pi^{\infty}}(\xi)q_{\Hc_\pi^{\infty}}(\xi')q_{\Hc_\pi^{\infty}}(\eta)q_{\Hc_\pi^{\infty}}(\eta')
\sup_{j\le i\le k} 
 \norm{p_j\varphi}_{L^\infty(G)}, 
 \end{aligned}
 $$
where $q_{\Hc_\pi^{\infty}}$ is the appropriate seminorm on $\Hc_\pi^{\infty}$.
Using \eqref{vaniv:1} with $\xi$, $\xi'\in \Hc_\pi^\infty$ fixed, depending on $\varphi$ such that
$\dual{\varphi}{{F}_{\xi, \xi'}}=1$, 
one gets 
$$
\begin{aligned}
\norm{p_j c_{\eta, \eta'}}_{L^\infty(G)} & \le C''' q_{\Hc_\pi^{\infty}}(\xi)q_{\Hc_\pi^{\infty}}(\xi')q_{\Hc_\pi^{\infty}}(\eta)q_{\Hc_\pi^{\infty}}(\eta') \, 
\sup\{\norm{p_i \varphi}_{L^\infty(G)}\mid j\le i\le k\}\\
& \le C(\varphi) q_{\Hc_\pi^{\infty}}(\eta)q_{\Hc_\pi^{\infty}}(\eta') \, 
\sup\{\norm{p_i \varphi}_{L^\infty(G)}\mid j\le i\le k\},
\end{aligned}
$$
which finishes the proof. 
 \end{proof}

\begin{definition}\label{vpidef}
Define the space of polynomial functions $V_\pi$ by
$$
 V_\pi:=\{p\in\Pc(G)\mid pc_{\xi,\eta}\in L^\infty(G),\,  \forall\,  \xi, \eta \in\Hc_\pi^\infty\}.
$$
 \end{definition}

 Note that $V_\pi$  is a left and right invariant subspace of $\Pc(G)$, since
 $$ (\lambda(x)p)c_{\xi, \eta} = \lambda(x) (p \lambda(x^{-1})(c_{\xi, \eta})) =
 \lambda(x)(p c_{\xi, \pi(x^{-1})\eta})$$
 and $\pi(x^{-1})\eta\in \Hc_\pi^\infty$, and similarly for $\rho(x) p$.

\begin{lemma}\label{lemma0}
Let $p\in \Pc(G)$ be arbitrary  and fixed. 
Then for every $N_0\in \NN$ we have
\begin{equation}\label{lemma0:ind}
 p((\ker\pi\cap \Sc(G))^{ N_{0}d_p})\subseteq (\ker{\pi}\cap \Sc(G))^{N_0}.
 \end{equation}
\end{lemma}

\begin{proof}
We prove \eqref{lemma0:ind} by induction on $d_p=\opn{dim}\Vc_p$.

If $d_p=1$, then $p$ is the constant function and \eqref{lemma0:ind} holds.  
Assume now that $d_p>1$.
Consider a Jordan-H\"older basis $\{p_{d_{p}}=p, \dots, p_{1}\}$ of $\Vc_p$. 
Let $f_i$, $i=1,\dots, d_p$ be elements of $(\ker\pi\cap \S(G))^{N_0}$.
We have seen (\ref{cpmulfg}) that 
$$
 p(f_{d_p}\ast\cdots\ast f_{1}  )=
\sum_{i=1}^{d_p}(\check a_{i,d_p}f_{d_{p}})\ast(p_i(f_{d_{p}-1} \ast \cdots\ast f_{1})).
$$
Here $a_{d_{p}, d_{p}}$ is the constant function 1 and $d_{p_i}<d_p$ for $i< d_p$. 
Hence by the induction hypothesis we know that 
$ p(f_{d_p}\ast\cdots\ast f_{1})\in (\ker{\pi}\cap \Sc(G))^{N_0}$, 
which proves \eqref{lemma0:ind}.
\end{proof}

The same proof gives the next corollary. 

\begin{corollary}\label{cor:lemma0}
 Let $p\in \Pc(G)$.
Then for every $N_0\in \NN$ we have
$$ p ((\ker{\pi} \cap C_0(G))^{N_0 d_p}) \subseteq (\ker{\pi})^{N_0}.$$
\end{corollary}

\begin{corollary}\label{pker}
For every $p\in \Pc(G)$  we have that 
$$p((\ker {d\pi))^{d_p N_\pi}\ast \Sc(G)}\subseteq j(\pi), $$
where $d_p$ is the $G$-degree of $p $.
\end{corollary}

\begin{proof}
We prove by induction that
for 
 $D_1, \dots, D_N\in \ker {d\pi}$ and $h\in C_0^\infty(G)$ 
we have
\begin{equation}\label{pker:1}
D_1\ast \cdots \ast D_N \ast h \in (\ker{\pi}\cap C_0(G))^N.
\end{equation}
When $N=1$ the assertion is clear. 

Assume now that we have proved \eqref{pker:1} for $N-1$ and all $N-1$ sets $D_1, \dots ,D_{N-1}\in \ker{d\pi}$. 
For every $k>0$ there are a finite number of $f_j\in C_0^k(G)$ and $v_j\in U(\gg_C)$ such that
$$\delta = \sum f_j\ast v_j.$$
 Then, by choosing $k$ such that $D_1\ast f_j \in C_0(G)$, we get 
$$ 
D_1\ast D_2\ast \cdots\ast D_N\ast h =\sum D_1\ast f_j \ast v_j \ast D_2\ast \cdots\ast D_N\ast h
$$
where the sum is finite. 
Since $v_k \ast D_2\in \ker{d\pi}$, we obtain
$$D_1\ast D_2\ast \cdots\ast D_N\ast h\in (\ker{\pi}\cap C_0(G))^N, $$
and this proves \eqref{pker:1}.

From Corollary~\ref{cor:lemma0}  and (\ref{npidetj}) we get that for  our $N_\pi\in \NN^*$ and
for every $p\in \Pc(G)$  we have
$$p((\ker {d\pi})^{N_\pi d_p}\ast C_0^\infty(G))\subseteq j(\pi)$$
Since $C_0^\infty(G)$ is dense in $\Sc(G)$, we get that 
$$p((\ker {d\pi})^{N_\pi d_p}\ast \Sc(G))\subseteq j(\pi).$$
\end{proof}

\begin{theorem}\label{vpiinj}
 The closed two-sided ideal 
$$
 K_\pi:=\{f\in L^1(G)\mid \dual{pc_{\xi,\eta}}{f}=0,\,  \forall p\in V_\pi,\, \forall  \xi,\eta \in \Hc_\pi^\infty \}
$$
contains the minimal ideal $j(\pi)$.
 \end{theorem}
\begin{proof}
Let
for $p\in V_\pi, p\ne 0$, $d_p$ be the dimension of   $\Vc_p$.
We have seen in the proof of Lemma~\ref{lemma0} that 
$$
 p((\ker\pi\cap \Sc(G))^{d_p})\subset \ker\pi.
$$
This shows that  $pc_{\xi,\eta}$ is contained in $((\ker \pi\cap\Sc(G))^{d_p})^\perp$ and therefore also 
$$
 \dual{pc_{\xi,\eta}}{(\ker\pi)^{d_p}}=\{0\}
$$
for every $\eta,\xi\in \Hc_\pi^\infty$. 
Hence
$$
 \dual{pc_{\xi,\eta}}{j(\pi)}=\{0\}
$$
since $j(\pi)=(\ker\pi)^{d}$ for every $d\geq N_\pi$.
\end{proof}

\subsection{A dense subspace of $j(\pi)$}\label{subsect2}

Consider 
$$ \gg= \gg_n \supset \gg_{n-1} \supset \cdots \supset \gg_1\supset \gg_0=\{0\}$$
a Jordan-H\"older sequence, $X_k \in \gg_k \setminus\gg_{k-1}$ a corresponding 
Jordan-H\"older basis, and $X^\ast_1, \dots, X^\ast_n$ its dual basis of $\gg^\ast$. 
Let $e$ be the set of  jump indices for $l$ and this basis, that is, 
$$ e=\{j \mid 1\le j\le n, \,  X_j \not\in \gg_{j-1} +\gg(l)\}. $$  
Denote  $\gg_e =\opn{span}\{X_j \mid j\in e\}$, and let $d= \opn{dim}{\gg_e}=\dim{\Oc}$. 
We then have $\gg_e +\gg(l)= \gg$. 
If $e'= \{ 1, \dots , n\} \setminus e$, then $\gg_{e'}= \opn{span}\{ X_j\mid j\in e'\}$ 
is also a complement of $\gg_e$, isomorphic with $\gg(l)$.
Let 
$e=\{j_1 <\dots < j_d\}$ be the set of jump indices, and denote 
$e'= \{ l_1< \dots < l_m\}$, $m+d=n$.  

A set of generators for $\ker{d\pi}$ were given in 
\cite{Pe84}:
The coadjoint orbit $\Oc_\ell$ is given by 
$$\Oc_\ell =\{ \sum\limits_1^n R_j(t) X_j^\ast\mid t\in \RR^d\},  $$
where 
the polynomials $R_j$ have the following properties:
\begin{itemize}
\item[i)] 
$R_j$ depends only on the variables $t_1, \dots , t_k$ where $j_k \le j < j_{k+1}$.
\item[ii)] $R_{j_k}(t) = t_k$, for all $k=1, \dots,d$.
 \end{itemize}
Then $\ker{d\pi}$ is generated by 
$u_j \in U(\gg_{\CC})$, where
\begin{equation}\label{uj}
 u_j = X_j - i \omega (R_j (- i X_{j_1}, \dots, -i X_{j_d})
\end{equation}
where $\omega\colon S(\gg_\CC) \to   U(\gg_{\CC})$ is the symmetrization map. 
Note that $u_j=0$ when $j\in e$.

\begin{definition}\label{Kpi0def}
We set 
$$
 K_{\pi, 0} :=\{f\in \Sc(G) \mid pf \in \ker{\pi}, \,  \forall\,  p\in \Pc(G) \}.
$$
\end{definition}

\begin{remark}
It is easy to see that $K_{\pi, 0}$ is a closed twosided ideal  of $\S(G) $ contained in  $\Sc(G)\cap \ker{\pi}$.
\end{remark}
\begin{lemma}\label{kpiconJ}
The subspace $K_{\pi,0} $ contains the generators of the  minimal ideal $J(\pi) $. 
 \end{lemma}
\begin{proof} We have seen in 
Definition~\ref{hulldef} that the minimal ideal $J(\pi) $ is generated by Schwartz functions $f $ for which exist $g\in\S(G) $ such that $\pi(g)=0 $ and $g\ast f=f\ast g=f $.
Hence, by Lemma~\ref{lemma0}  for $p\in\PP(G) $ we have that 
$$
 pf=p(g^{ d_p }\ast f)\subset \ker\pi,  
$$
therefore $f\in K_{\pi, 0}$. 
\end{proof}

\begin{lemma}\label{Kpi0}
The subspace $K_{\pi, 0}$ is contained in $j(\pi)$.
\end{lemma}

\begin{proof} 
Let $X_1, \dots, X_n$ be the Jordan-H\"older basis as above, and recall that we have identified $G$ and $\gg$ 
via the exponential mapping. 
 
On $G\simeq \gg$ we consider the chart
\begin{equation}\label{chart}
\theta\colon  \exp{\sum\limits_{j=1}^n x_j X_j} \mapsto (x_1, \dots, x_n)\in \RR^n.
\end{equation}
Thus we may assume that $G=\RR^n$ with multiplication given by
$$ x\cdot y = \theta (\theta^{-1}(x)\cdot \theta^{-1}(y)).  $$
We consider polynomials on $G$ to be written in the chart $\theta$, thus we will indentify a $p\in \Pc(G)$ with 
$ p\circ \theta^{-1}$ which is a polynomial function on $\RR^n$.  

Recall that for $j=1, \dots, n$, and $\varphi\in C^\infty(G)$ we have
\begin{equation}\label{Xj}
\begin{aligned}
(X_j \ast \varphi)(y) &  = (d \lambda(X_j)\varphi)(y) = \frac{d}{ds} \varphi((-sx)\cdot y)\vert_{s=0}\\
& = -\partial_{k}\varphi(y)  +\sum\limits_{j=1}^{k-1} \alpha_{jk}(y_{j+1}, \dots, y_n)\partial_j \varphi(y),  
\end{aligned}
\end{equation}
where $\alpha_{jk}$ are polynomial functions.

Let $f \in K_{\pi, 0}$. 
Then $f\in \ker{\pi}\cap \Sc(G)= \ker{d\pi}\ast \Sc(G)$ (\cite[Thm.3.5]{dC87}), hence there are Schwartz functions $g_1, \dots, g_m$ such that 
$$ f = u_{l_m} \ast g_m +\cdots + u_{l_1}\ast g_1. $$
Let $p$ a polynomial that depends on $x_{l_m}$ only. 
Then by \eqref{uj} and \eqref{Xj} we have that 
$$ pf = - p'g_{m} + u_{l_m}\ast (pg_m) + \sum \limits_{j< m} u_{l_j} \ast (pg_j).$$
If we chose $p(x)=p_1(x)= x_{l_m}$, since $pf\in \ker{\pi}$, 
we obtain that 
$$ g_m = \sum\limits_{j\le m} u_{l_j}\ast h_j  \in \ker{\pi}.$$ 
Thus it follows that
$$ f = u_{l_m}^2 \ast h_{m} + \sum\limits_{j< m} u_{l_m} \ast u_{l_j} \ast h_j + 
\sum\limits_{j< m} u_{l_m} \ast g_j. $$
Since 
$u_{l_m} \ast u_{l_j} \in \sum\limits_{k<m} u_{l_k}\ast U(\gg_\CC)$ 
when $j<m$ (see \cite[Prop~2.3.1]{Pe84}), 
we can write
\begin{equation}\label{Kpi0:1}
f= u_{l_m}\ast u_{l_m}\ast g_{m, 2} + \sum\limits_{j<m} u_{l_j} \ast g_{j, 2}, 
\end{equation}
for some
$g_{k, 2}\in \Sc(G)$.

We prove by induction  that for each $N$ there are $g_{k, N}\in \Sc(G)$ such that 
\begin{equation}\label{ulmN}
 f = u_{l_m}^N \ast g_{m, N} +\sum\limits_{j<m} u_{l_j} \ast g_{j, N}.
\end{equation}
Since  $f$ has been chosen arbitrary and $pf$ satisfies the same properties as $f$ for all $p\in \Pc(G)$ 
we will also have  also that
$$
pf \in u_{l_m}^N \ast \Sc(G) + \sum\limits_{j<m} u_{l_j}\ast \Sc(G).$$

We show first by induction over $k\in \NN$ 
\begin{equation}\label{Kpi02}
x_{l_m}^k(u_{l_m}^k\ast g) = (-1)^k k! g + u_{l_m}\ast h_k
\end{equation} 
for all $g\in \Sc(G)$, where $h_k\in \Sc(G)$.
For $k=1$ this has been shown above.
If $k>1$ denote $g_{k-1} = u_{l_m}^{k-1}\ast g$.
Then
$$
u_{l_m}\ast (x_{l_m}^k g_{k-1})= k x_{l_m}^{k-1}  g_{k-1} + x_{l_m}^k (u_{l_m}^k\ast g).$$
Hence by the induction hypothesis, 
$$
\begin{aligned}
 x_{l_m}^k (u_{l_m}^k\ast g) & = u_{l_m}\ast (x_{l_m}^k g_{k-1}) - k[ (-1)^k (k-1)! g + u_{l_m} \ast h_{k-1}]\\
& = u_{l_m} \ast h_k + (-1)^k k! g.
\end{aligned}
$$
This proves \eqref{Kpi02}.

Now assume that we have shown that
$$ f = u^k_{l_m} \ast g_{m, k} +\sum \limits_{j<m} u_{l_j} \ast g_{j, k}$$
for some $g_{j, k}\in \Sc(G)$. 
Then if $p_k(x) = x_{l_m}^k$, we have by using \eqref{Kpi02}, 
$$ p_kf = (-1)^k k! g_{m, k} + \text{term in}\, \ker{\pi}, $$
hence
$g_{m, k} \in \ker{\pi}$. 
Thus \eqref{ulmN} follows. 

Denote by $\Pc_{k}$ polynomials in variable $x_{l_k}$, $1\le k\le m-1$, and by  $\Pc_{k}^D$  the subset of $\Pc_k$ consisting of polynomials of degree $D$.
Iterating the above procedure we get that for every $N_m$, $N_{m-1}, \dots, N_1$ we get that
\begin{equation}\label{mainform}
\begin{aligned}
 f \in & \Pc_{1}^{\widetilde{N_{1}}}\cdots \Pc_{m-1} ^{\widetilde{N_{m-1}}} ((\ker{d\pi})^{N_m}\ast \Sc(G)) \\
& + \Pc_{1}^{\widetilde{N_1}}\cdots \Pc_{m-2}^{\widetilde{N_{m-2}}} ((\ker{d\pi})^{N_{m-1}} \ast \Sc(G))+ \cdots \\ 
 & \cdots + \Pc_{1}^{\widetilde{N_1}}((\ker{d\pi})^{N_2}\ast \Sc(G)) + 
(\ker{d\pi}) ^{N_1}\ast \Sc(G) 
\end{aligned}
\end{equation}
for all $f\in K_{\pi, 0}$,  where $\widetilde{N_k} = 1+ 2+\cdots+(N_k-1)$. 
Indeed, assume that we have proved that for all $f\in K_{\pi, 0}$ and $1< k<  m$, 
$$ 
f = \sum\limits_{j\le k} u_{l_j}\ast g_{j} +  \Pc_{k+1}^{\widetilde{N_{k+1}}}\cdots \Pc_{m-1}^{\widetilde{N_{m-1}}} ((\ker{d\pi})^{N_m}\ast \Sc(G)) 
+ \cdots 
 + (\ker{d\pi})^{N_{k+1}} \ast \Sc(G) 
$$
and thus $pf$ can be written similarly for every $p\in \Pc(G)$. 
We have on the other hand 
$$
\begin{aligned} 
x_{l_k} f &  = - g_k + u_{l_k}\ast (x_{l_k} g_k) + \sum\limits_{j<k} u_{l_j} \ast (x_{l_k}g_j) \\
& +  x_{l_k} P_{k+1}^{\widetilde{N_{k+1}}} \cdots \Pc_{m-1}^{\widetilde{N_{m-1}}} ((\ker{d\pi})^{N_m}\ast \Sc(G))+\cdots+ x_{l_k} 
 (\ker{d\pi})^{N_{k+1}} \ast \Sc(G))
\end{aligned}
$$
It follows that $g_k$ is of the form
$$
\begin{aligned}
g_k & = u_{l_k}\ast h_k + \sum\limits_{j<k} u_{l_j}\ast h_j \\
& +  x_{l_k} P_{k+1}^{\widetilde{N_{k+1}}} \cdots \Pc_{m-1}^{\widetilde{N_{m-1}}} (\ker{d\pi})^{N_m}\ast \Sc(G))+\cdots+ x_{l_k} 
 ((\ker{d\pi})^{N_{k+1}} \ast \Sc(G)), 
\end{aligned} 
$$
with $h_j\in \Sc(G)$. 
Replacing this in the formula for $f$ we get that
$$
\begin{aligned}
f &=   u_{l_k}^2\ast  h + \sum\limits_{j<k} u_{l_k} \ast u_{l_j} \ast h_j \\
 + &  \Pc_{k+1}^{\widetilde{N_{k+1}}} \cdots \Pc_{m-1}^{\widetilde{N_{m-1}}} (\ker{d\pi})^{N_m}\ast \Sc(G))+\cdots +(\ker d\pi)^{N_{k+1}} \ast \Sc(G)\\
+ &  u_{l_k}\ast (x_{l_k}  \Pc_{k+1}^{\widetilde{N_{k-1}}} \cdots \Pc_{m-1}^{\widetilde{N_{m-1}}} (\ker{d\pi})^{N_m}\ast \Sc(G))+\cdots + x_{l_k} 
 ((\ker{d\pi})^{N_{k-1}} \ast \Sc(G)).
 \end{aligned}
 $$
 Note that $u_{l_k} \ast (pg) = p(u_{l_k} \ast g)$ when $p\in \Pc_{j}$, $j>k$.
 Hence
 $$ 
\begin{aligned}
 f & = u_{l_k}^2 \ast g_{k, 2} + \sum\limits_{j<k} u_{l_j}\ast g_{j, 2}\\
& + \Pc_k^1[P_{k+1}^{\widetilde{N_{k+1}}} \cdots \Pc_{m-1}^{\widetilde{N_{m-1}}} (\ker{d\pi})^{N_m}\ast \Sc(G))+\cdots+ 
 (\ker{d\pi})^{N_{k+1}} \ast \Sc(G))] 
\end{aligned} 
$$
Assume that we have proved that for all $f\in K_{\pi, 0}$ we have  
\begin{equation}\label{1001}
\begin{aligned}
 f& =  u_{l_k}^q \ast g_{k, q} + \sum\limits_{j<k} u_{l_j}\ast g_{j, q}\\
&+   \Pc_k^{\widetilde{q}}[ P_{k+1}^{\widetilde{N_{k+1}}} \cdots \Pc_{m-1}^{\widetilde{N_{m-1}}} (\ker{d\pi})^{N_m}\ast \Sc(G))+\cdots+ 
 (\ker{d\pi})^{N_{k+1}} \ast \Sc(G))] 
 \end{aligned}
\end{equation}
 Then $x_{l_k}^q f$ can be written similarly.
Also \eqref{1001} shows that  for some $\tilde{g}_{k, q}\in \S(G)$,  
 $$
 \begin{aligned}
 x_{l_k}^q f & = (-1)^{q}q! g_{k, q} +  u_{l_k}^2 \ast {\tilde{g}}_{k, q} + 
 \sum\limits_{j<k} u_{l_j}\ast (x_{l_k}g_{j, q})\\
 & +
 x_{l_k}^q \Pc_k^{\widetilde{q}}( P_{k+1}^{\widetilde{N_{k+1}}} \cdots \Pc_{m-1}^{\widetilde{N_{m-1}}} (\ker{d\pi})^{N_m}\ast \Sc(G))+\cdots
\cdots + 
 (\ker{d\pi})^{N_{k+1}} \ast \Sc(G)).
\end{aligned}
$$
Hence $g_{k, q}$ is of the form
$$
\begin{aligned}
 g_{k, q}& = u_{l_k}^2\ast h_k+ \sum\limits_{J<k} u_{l_j}\ast h_j \\
& + x_{l_k}^q \Pc_k^{\widetilde{q}}[P_{k+1}^{\widetilde{N_{k+1}}} \cdots \Pc_{m-1}^{\widetilde{N_{m-1}}} (\ker{d\pi})^{N_m}\ast \Sc(G))+\cdots+ 
 (\ker{d\pi})^{N_{k+1}} \ast \Sc(G))].
\end{aligned}
$$
Replacing this in \eqref{1001}, and using the fact that
$$ 
\begin{aligned}
& u_{l_k}^q \ast (x_{l_k}^q \Pc_k^{\widetilde{q}}[\Pc_{k+1}^{\widetilde{N_{k+1}}} \cdots \Pc_{m-1}^{\widetilde{N_{m-1}}} (\ker{d\pi})^{N_m}\ast \Sc(G))+\cdots+ 
 (\ker{d\pi})^{N_{k+1}} \ast \Sc(G))] \\
& \subseteq \Pc_k^{\widetilde{q}+q }[\Pc_{k+1}^{\widetilde{N_{k+1}}} \cdots \Pc_{m-1}^{\widetilde{N_{m-1}}} (\ker{d\pi})^{N_m}\ast \Sc(G))+\cdots+ 
 (\ker{d\pi})^{N_{k+1}} \ast \Sc(G))
\end{aligned}
$$
 we get that \eqref{1001} holds for $q$ replaced by $q+1$. 

We have thus proved \eqref{mainform}.

Now we choose 
\begin{itemize}
\item{} $N_1$ large enough such that $(\ker{d\pi})^{N_1} \ast \Sc(G) \subseteq j(\pi)$;
\item{} $N_2$ large enough such that $\Pc_{1}^{\widetilde{N_1}}((\ker{d\pi})^{N_2} \ast \Sc(G))\subseteq j(\pi)$;
\item{} $\cdots$
\item{} $N_m$ large enough such that 
$\Pc_{1}^{\widetilde{N_{1}}}\cdots \Pc_{m-1}^{\widetilde{N_{m-1}}} ((\ker{d\pi})^{N_m}\ast \Sc(G))\subseteq j(\pi)$.
\end{itemize}
Hence we eventually get that $f \in j(\pi)$.
\end{proof}

Inspecting the proof of Lemma~\ref{Kpi0} one sees that it gives us 
 in fact an  $N_0\in\NN$  such that the closed subspace of $\Sc(G)$ defined by
$$ \{ f\in \Sc(G)\mid pf \in \ker{\pi}; \; \forall\, p\in \Pc(G), \; \opn{dim}\Vc_p \le N_0, p \; \text{real}\}
$$
is contained in $j(\pi)$. 
This, along with Lemma~\ref{kpiconJ}, proves  the following fundamental theorem.  

\begin{theorem}\label{Jpi0cor}
There is finite dimensional subspace $\PP_0^\pi=\Pc_0$ of $\Pc(G)$,  invariant under the action of $G$ by left and right translations, such that the closed ideal $ J_{\pi, 0}$ in $\Sc(G)$, defined by
$$ J_{\pi, 0}=\{ f\in \Sc(G)\mid pf \in \ker{\pi}; \; \forall\, p\in \Pc_0\},
$$
is contained in $j(\pi) \cap \Sc(G)$ and it is dense in $j(\pi)$.
\end{theorem}

\section{Ideals with hull $\{\pi\}$} 

\subsection{Maximal projections}\label{minproj}

The following lemma goes along the same lines as \cite[Lemma~7.6]{LMP13}.
\begin{lemma}\label{min-proj}
Let $J$ be a closed ideal in $\Sc(G)$ such that $J\subseteq \ker{\pi} \cap \Sc(G)$ and assume that there is 
$d\in \NN$ such that
\begin{equation}\label{nilJ}
(\ker{\pi}\cap \Sc(G)/J)^d=\{ 0\}.
\end{equation}
Let $q$ be a projection in $\Sc(G)/(\ker{\pi}\cap \Sc(G))$.
Then there exists $\mathbf{q}\in \Sc(G)/J$ such that $\mathbf{q}\ast \mathbf{q} =\mathbf{q}$, $\mathbf{q}^* =\mathbf{q}$, 
and,  if 
$$ \mu\colon \Sc(G)/J\longrightarrow \Sc(G)/(\ker{\pi}\cap \Sc(G))= \big(\Sc(G)/J)\big)/ 
\big((\ker{\pi} \cap \Sc(G))/J \big)
$$
is the canonical projection, then $\mu(\mathbf{q})= q$.
\end{lemma}

\begin{proof}
Let $q\in  \Sc(G)/(\ker{\pi}\cap \Sc(G))$ be fixed. 
Then consider $g\in \Sc(G)/J$ such that $g^\ast = g$ and $\mu(g) = q$. 
Since $\ker{\mu} = (\ker{\pi} \cap \Sc(G))/J$ and \eqref{nilJ} holds,  it follows that
$$ (g -g\ast g)^d = 0 \quad \text{in} \quad \Sc(G)/J.$$

On the other hand, by Bezout's theorem,  
it is easy to see that there exists two polynomials $\Psi, \Phi\in \QQ[t]$ such that
$1 = t^d \Psi(t) + (1-t)^d \Phi(t)$. 
Hence 
\begin{equation}\label{min-proj:1}
 (t^d \Psi(t))^2 = t^d \Psi(t) - (t-t^2)^d \Psi(t)\Phi(t).
\end{equation}
Denote $\psi(t) = t^d \Psi(t)$ and $\mathbf{q}:= \psi(g) \in  \Sc(G)/J$.
Then using \eqref{min-proj:1} in the commutative subalgebra generated by $g$ we get that
$$ \mathbf{q}\ast \mathbf{q}  = \mathbf{q} \quad \text{in} \quad \Sc(G)/J.$$
Also, $\mathbf{q}$ is self-adjoint since $\psi$ is real and $g$ self-adjoint.

Now, writing  $\Psi(X)=\sum_{i=0}^n \ps_i X^i$, we see that
$$
 1=1^d\Psi(1)+0=\sum_{i=0}^n \ps_i.
$$
Furthermore, since $\mu(g^k) =\mu(g)=q$ for every $k\ge 2$, 
we get that 
$$ \mu(\mathbf{q}) = q^{\ast d}\ast  \Psi(q) \ = q\ast (\sum_{i=0}^n\ps_i q^d)=(\sum_{i=0}^n\ps_i)q=q.$$ 
\end{proof}

\begin{remark}\label{Qxi-rem}
Recall that the ideal $J_{\pi, 0}$ in Corollary~\ref{Jpi0cor} is closed in $\Sc(G)$ and contained in $j(\pi)\cap \Sc(G)$. 
Moreover, since $\Pc_0$ is finite dimensional, Lemma~\ref{lemma0} shows that there is a $d>0$ such that 
$$ (\ker{\pi}\cap \Sc(G))^d \subseteq J_{\pi, 0}.$$ 
Hence we can use Lemma~\ref{min-proj} for $J_{\pi,0}$.
Thus for every $\xi\in \Hc_\pi^\infty$ we can find $Q_\xi=Q_{\xi}^\ast\in \Sc(G)$ such that 
\begin{equation}\label{Qxi}
Q_{\xi}\ast Q_{\xi} =Q_{\xi} \, \text{mod}\, J_{\pi, 0}, \quad \pi(Q_{\xi}) = P_{\xi, \xi}= \scalar{\cdot}{\xi}\xi.
\end{equation}
\end{remark}

\begin{definition}\label{maximalpro}
\normalfont  
Let $\xi $ be a smooth vector in $\Hc_\pi$. 
An element $Q=Q_\xi\in \S(G) $ is called  a \textit{maximal projection} (for $\xi $), if $Q=Q^*$ 
is a positive element in the $C^\ast$ algebra  $C^*(G) $ such that  $\pi(Q)=P_{\xi,\xi} $, and  
$\pi'(Q)\ne 0 $ for any $\pi'\in\widehat G $.
 \end{definition}

\begin{proposition}\label{exmax}
Let $\xi\in \H_\pi^\infty $. 
A  maximal projection $Q=Q_{\xi} $ always exists  
 such that $Q\ast Q- Q\in J_{\pi, 0}\subset j(\pi)$.
\end{proposition}
\begin{proof}
Choose any $Q'=(Q')^*\in \S(G) $ as in Remark~\ref{Qxi-rem}, with $\pi(Q')=P_{\xi, \xi}$,  and take $Q_1= Q' \ast Q'$. 
There exists by \cite[Prop.~4.7.4]{Di77} an element  $D\in\U(\g) $, such that $d\pi(D)=0 $ and 
$d\pi'(D)\ne 0 $ for any $\pi'\in\widehat G, \pi'\ne \pi $. 
Let $h_t$, $t>0,$ be any heat kernel 
function for  some full Laplacian on $G$. 
The operator $\pi'(h_t) $ has dense range, and therefore, for an $N$ is large enough, the function $f:=((D\ast h_t)\ast (D\ast h_t)^*)^{\ast N_\pi} \in \S(G)$
has the property that $f \in J_{\pi, 0}(\pi) $ and $\pi'(f)\ne 0$ for $\pi'\ne \pi $. 
Then $Q:=Q_1+f $ satisfies the conditions in the statement.
\end{proof}

\begin{remark}\label{TheQ}
For the rest of the paper, when $\xi \in \Hc_\pi^\infty$ is fixed, $Q=Q_\xi$ will be a maximal projection 
satisfying the conditions of Proposition~\ref{exmax}.
\end{remark}

\subsection{A correspondence between ideals with hull $\pi$ and the ideals of a finite dimensional algebra}
\begin{definition}\label{defjhpn}
\normalfont 
\begin{enumerate}
\item\label{defjhpn_1} 
Let $\II^\pi $ be the set of all closed two-sided ideals $I $ in $L^1(G) $ with hull  $h(I)=\{\pi\} $.
\item\label{defjhpn_2}
Let also $\Cc_\pi$
be  the vector space of all smooth coefficient functions $c_A, A\in \mathbb B(\H_\pi)_\iy $.  
\item\label{defjhpn_3}  
Let $p\in\PP_0 $ and let   $\{p_m,\dots, p_1=1\} $ be a Jordan-H\"older basis  of  the  $\lambda\otimes\rho $  sub-module  
$\Vc_p $ of $\PP_0 $ generated by $p$.
Then   for  $j\in \{0,\dots, m\}$,  let $\V_p^jC_\pi $ 
be the space of all  functions of the form $\sum_{i=1}^j p_i \varphi_i $, where the $\varphi_i$, $i=m,\dots, 1$ are  functions in $\Cc_\pi $. 
 \end{enumerate}
 \end{definition}
For $ \varphi\in \ker\pi^\perp$,  $\de,\eta \in \Hc_\pi^\infty$ of norm $1$, and  $Q_\et,Q_\de$ as in Remark~\ref{TheQ},
let us compute $Q_\de\ast(p_j\varphi)\ast Q_\et$ using (\ref{cpmulfg}), (\ref{trnscoeff}) and (\ref{vaniv:1}), as follows:
\begin{eqnarray}\label{qastcexp}
 \check Q_\de\ast(p_j\varphi)\ast \check Q_\et&=&
\sum_{i=1}^j(p_i((a_{i,j} \check Q_\de)\ast \va))\ast \check Q_\eta 
\nonumber
\\
&=&( p_j(\check Q_\de\ast\va))\ast \check Q_\et+ \sum_{i=1}^{j-1} (p_i((\check a_{i,j}\check Q_\de)\ast \va))\ast \check Q_\eta
\nonumber
\\
&=&\sum\limits_{i=1}^{j}p_i(\check Q_\de\ast \va \ast b_{i,j}\check Q_\eta)
+ \sum\limits_{i=1}^{j-1} \sum_{k=1}^{i}p_k((\check a_{i,j}\check Q_\de)\ast \va\ast (b_{k,i}\check Q_\eta)) 
\label{hert} \\
&=&\langle{\de},{\eta}\rangle \langle{\va},{\check Q_\de\ast \check Q_\et}\rangle p_jc_{\et,\de}
+\sum\limits_{i=1}^{j-1}p_i \varphi_i
\nonumber
 \end{eqnarray}

for some $\varphi_{j-1},\dots, \varphi_1\in \ker\pi^\perp $.
Furthermore, since $ p_j\va\in J_{\pi,0}^\perp $, 
$$
 \check Q_\de\ast p_j\va\ast \check Q_\et=\check Q_\de\ast \check Q_\de\ast(p_j\varphi)\ast \check Q_\et\ast \check Q_\et.
$$
This shows inductively that
 \begin{equation}\label{pinguin}
 \check Q_\de\ast p\varphi\ast \check Q_\et = \check Q_\de \ast \tilde{p}c_{\et, \de} \ast \check Q_\et, 
 \end{equation}
for some $\tilde{p}=p+p'\in \Vc_p$ with $p'\in \V_p^{m-1} $.

\begin{theorem}\label{expjpi}
Let $\de,\et\in\H_{\pi}^\iy $ of norm 1 and $I\in\II^\pi $. Then for any $\ps\in I^\perp$ there exists $p\in \PP_0 $ such that
$$
 \check Q_\de\ast \ps\ast \check Q_\et=\check Q_\de\ast (pc_{\et,\de})\ast \check Q_\et
 $$
on $\S(G) $.  
 \end{theorem}

\begin{proof}
Let $\omega\colon G\mapsto [1,\infty[ $ be a smooth symmetric polynomial weight, such that $\PP_0 $ is bounded by $\omega$.
(Such a weight always exists, see \cite{LM10}.)  
Then the sub-space $\PP_0 \ker\pi^\perp $ spanned by the functions $p\varphi$, $p\in\PP_0$, $\varphi\in\ker\pi^\perp $ is contained in $L^\infty(G,\om)$ 
and 
$$
j(\pi)\cap L^1(G,\omega)\supset\{f\in L^1(G,\omega)\mid \dual{\ps}{f}=0, \, \psi \in \PP_0 \ker\pi^\perp\}\supset J_{\pi,0}.
$$

Let now $\psi\in j(\pi)^\perp$. 
Then $ \ps$ is contained $L^\infty(G,\omega) $ and thus in the weak$ ^*$-limit of 
$\text{span}\{p\varphi \mid  p\in\Pc_0,\varphi \in \ker \pi^\perp\}\subset L^\infty(G,\omega)$.
Hence  $ \check Q_\de\ast \ps\ast \check Q_\eta$ is  contained in the weak$^* $ closure of 
$ \check Q_\de\ast (\Pc_0 c_{\et,\de})\ast {\check Q}_\eta$  and therefore is contained in $ \check Q_\de\ast (\Pc_0 c_{\et,\de})\ast \check Q_\eta$,  because this is  a finite dimensional space.
\end{proof}

\begin{remark}\label{qpsq}
\normalfont
Let $ \psi\in j(\pi)^\perp$ and $ \de,\et\in \H_\pi^\infty$. 
Then there exists $ p\in\PP_0$, such that 
\begin{eqnarray}\label{qastpspc}
\check{ Q}_\de\ast \ps\ast {\check Q}_\eta =pc_{\et,\de}+\sum_{i=1}^{m-1}p_i\ps_i
 \end{eqnarray}
where $\ps_i\in \Cc_\pi $ and where $ \{p=p_m,\cdots, p_1\}$ is  a Jordan-H\"older basis of $ \V_p$.
  Indeed by Theorem \ref{expjpi}, we have this $ p\in\PP_0$ such that $ {\check Q}_\de\ast\ps\ast 
  {\check Q}_\eta={\check Q}_\de\ast pc_{\et,\de}\ast {\check Q}_\eta$. 
  Then by (\ref{hert}) we have
\begin{eqnarray}\label{eeqppsq}
\nonumber \check Q_\de\ast\ps\ast \check Q_\eta &= & \check Q_\de\ast pc_{\et,\de}\ast \check Q_\et\\
&=& \sum_{i=1}^{j}p_i(\check Q_\de\ast c_{\et,\de} \ast b_{i,j}\check Q_\et)
+ \sum_{i=1}^{j-1} \sum_{k=1}^{i}p_k((\check a_{i,j}\check Q_\de)\ast c_{\et,\de}\ast (b_{k,i}\check Q_\et))
\\
&=&\sum_{i=1}^{j}p_ic_{\pi(b_{i,j}\check Q_\eta)\eta,\pi( Q_\de)^*\de} )
+ \sum_{i=1}^{j-1} \sum_{k=1}^{i}p_k c_{\pi(\check b_{k,i} Q_\et)\eta, \pi( a_{i,j} Q_\de)^*\de,})
\nonumber
 \end{eqnarray}
 \end{remark}

\begin{corollary}\label{findim}
Let $\xi\in \H_\pi^\infty$ of norm $1$, and $Q=Q_\xi$.
The vector space $Q\ast (\ker\pi/j(\pi)\ast Q$ is finite dimensional.
 \end{corollary}
\begin{proof} 
We have seen in Theorem~\ref{expjpi} that 
$$\check Q\ast \PP_0c_{\xi,\xi}\ast \check Q\supset \check Q\ast j(\pi)^\perp\ast \check Q
=( Q\ast (\ker\pi/j(\pi))\ast  Q)^* $$ is finite dimensional. 
Hence the space $Q\ast (\ker\pi/j(\pi))\ast Q$ is itself finite dimensional.
\end{proof}

\begin{definition}\label{Qgen}
\normalfont 
Let $ Q=Q_\xi$ be  as in Remark~\ref{TheQ}.
\\
i) We denote by  $ M^Q$  the span of the subset $ \{f\ast Q\ast h\mid  h,f\in \L1G\}$. 
\\
ii) For every two-sided closed ideal $ I$ in $ \II^\pi$, let 
\begin{equation}\label{IQ}
 I_Q:=Q\ast (I/ j(\pi))\ast Q.
 \end{equation}
Then $ I_Q$ is an  ideal in the finite dimensional algebra $ \ker\pi_Q=Q\ast \ker \pi/j(\pi)\ast Q$.
\end{definition}

\begin{proposition}\label{mqdense}
The subspace $ M^Q$ is dense in $L^1(G)$ modulo $\ker\pi$. 
 \end{proposition}
\begin{proof}
Indeed, let $f\in \S(G)$ and $\ep>0 $. 
There exists an $N\in\NN $ large enough, such that if  $h=h^* \geq 0 $ is any self-adjoint 
smooth function of compact support of norm $\no h_1=1 $ on $G $,  then there  exists a smooth real 
valued compactly supported function $\psi$ on $\R$ vanishing in a neighbourhood of $0 $, such that
$$
 \no{\ps\{h\}-h^{\ast N}}_1<\ep, 
$$
where $\psi\{h\}=(2\pi)^{-1} \int_\R \hat \psi(t)e^{\ast it h}\, dt.$
 The function $\psi\{f\}$ is contained in the minimal ideal $j(\emptyset) $ 
 and has the property that $\pi(\psi\{h\}) $ is a self-adjoint  finite rank operator. 
We can now choose $h$ as above such that $\norm{\psi\{h\}\ast f-f}_1<\ep $. (See \cite{Di60}.) 
The operator  $\pi(\psi\{h\}\ast f) $ is smooth and of finite rank. 
Hence we can write
$$
 \pi(\psi\{h\}\ast f)=\sum_{\text{finite}}\pi(f_i\ast Q\ast g_i)
$$
for some finite family $\{f_i,g_i\} $ of Schwartz-functions. Let 
$m:=\sum_{\text{finite}}f_i\ast Q\ast g_i $. 
Then $k:= \psi\{h\}\ast f-m\in\ker\pi$ and 
$$
 \norm{f-m -k}_1<\ep.
$$
\end{proof}

\begin{proposition}\label{mqgene}
Let $\xi \in \Hc_\pi^\infty$ of norm $1$, $Q=Q_\xi $.
Then for every  $I\in \II^\pi $ we have 
$$
 I=\overline{\opn{span}\{L^1(G)\ast Q\ast I\ast Q\ast L^1(G)\}}
$$
\end{proposition}

\begin{proof}
Let $ \varphi\in \{L^1(G)\ast Q\ast I\ast Q\ast L^1(G)\}^\perp$.  
It is enough to show that $ \dual{\varphi}{I} =\{0\}$.

Let $ 0\ne f\in I$ of norm 1. 
For $ \ep>0$, we choose $ h\in L^1(G)$ such that $ \norm{h\ast f-f}_1<\ep$. 
There is $ h_0\in\ker\pi$ and $ m_1\in M^Q$ such that 
$$
 \norm{h-h_0-m_1}_1< \ep.
$$
Then
$$
 \norm{f-m_1\ast f-(h_0\ast f)}_1<2\ep,
$$
Continuing in this way, we find $ h_1\in\ker\pi$ and $ m_2\in M^Q$ such that 
$$
 \norm{h_0\ast f-m_2\ast h_0\ast f-h_1\ast h_0\ast f}_1<3\ep.
$$

Repeating the above argument $ N_\pi$-times we see that $ f$ can be approximated by elements in $ M^Q\ast f$ modulo $(\ker\pi)^{{N_\pi}}$, hence modulo $j(\pi)$.
The same approximations on the right eventually show that $f\in \ol{M_Q\ast I \ast M_Q + j(\pi)}$. 
Now the  closed ideal $\ol{M_Q\ast j(\pi)\ast M_Q} $ is contained in $j(\pi) $ and its hull is reduced to $\{\pi\} $, 
since $Q $ is maximal. 
Hence $\ol{M_Q\ast j(\pi)\ast M_Q}=j(\pi) $ and thus we get that
$ f\in \ol{M_Q\ast I\ast M_Q}$. 
Since 
$$
\dual{\varphi}{M^Q\ast I\ast M^Q} =\{0\}
$$
it follows that $ \dual{\varphi}{f}=\{0\}$ for all $ f\in I$. 
We obtain thus  that $\varphi$ vanishes on $I$.  
\end{proof}

\begin{theorem}\label{idetiq}
Let $\xi\in \H_\pi^\infty $  and let  $Q=Q_\xi \in \Sc(G)$ be a maximal projection like in Remark~\ref{TheQ}. 
The mapping $ I\mapsto I_Q$ from the space $ \mathcal I^\pi$ of the primary twosided ideals $ I$ with hull $ \pi$ into the space $\I^Q $ of  twosided ideals of the algebra  $ \ker\pi_Q$ is a bijection.
 \end{theorem}
 \begin{proof} Let $I$, $I'\in \II^\pi $ be  such that $Q\ast I/j(\pi) \ast Q= Q\ast I'/j(\pi)\ast Q $. 
Then by the proposition above $I=I'$. 
 
Let $K$ be a two-sided ideal in $\ker\pi_Q$. 
 Define the ideal $I_K $ by
 $$
 I_K:=\overline{\opn{span}\{L^1(G)\ast \tilde K\ast L^1(G)\}},
 $$
where 
$\tilde K:=\{f\in Q\ast L^1(G)\ast Q \mid  f\text{ mod } j(\pi)\in K \}$. 
Then $I_K $ is a closed twosided ideal in $L^1 (G) $. 
Furthermore the hull of this ideal is reduced to $\{\pi\} $, since $Q $ is maximal.
Also, since $K $ is an ideal,  we clearly have that $K= Q\ast (I_K/j(\pi)) \ast Q=(I_K)_Q$. 
\end{proof}

\begin{theorem}\label{sdense}
For every ideal $ I\in\II^\pi$, the subspace $ I\cap \S(G)$ is dense in $ I$.
 \end{theorem}
 \begin{proof} 
 Choose $Q=Q_\xi \in\S(G)$ as above. 
 Then there is an ideal $K$ in $(\ker\pi)_Q$ such that $I=I_K$ as in the proof of Theorem~\ref{idetiq}. 
We also have that $\ker\pi \cap \Sc(G)$ is dense in $\ker\pi$, and $j(\pi) \cap \Sc(G)$ is dense in $j(\pi)$ (see \cite{Lu83b}).  
Then, since  $\ker\pi_Q$ is finite dimensional, we have  $\ker\pi_Q=(\ker\pi)\cap \S(G))_Q $. 
Therefore, denoting by $\tilde K $ the subspace of $Q\ast L^1(G\ast Q $ defined by
\begin{eqnarray*}
 \tilde K:=\{f\in Q\ast L^1(G)\ast Q\mid f\text{ mod }j(\pi)\in K\}
 \end{eqnarray*}
 we see that every $f\in\tilde K $ can be approximated by functions contained in $\tilde K\cap \S(G) $. Hence  every element in  $I=\ol{L^1(G)\ast \tilde K\ast L^1(G)} $ can also be approximated by  functions contained in $\S(G)\ast (\tilde K\cap \S(G))\ast \S(G) $.
\end{proof}

\begin{definition}\label{defpI}
\normalfont  Let $I\in\II^\pi $, $\xi\in \H^\infty_\pi $ and $Q=Q_\xi$ as before. 
 Let 
$$
 \PP_Q^I:=\{p\in\PP_0\mid  \text{there exists }\psi\in I^\perp\text{ such that }\check Q\ast \ps\ast  \check Q=\check Q\ast pc_{\xi,\xi}\ast  \check Q\text{ on } \S(G)\}.
 $$
  Let also
 $$
 \PP^\pi_Q=\PP^{j(\pi)}_Q.
 $$
 \end{definition}
 \begin{remark}\label{can be n}
\normalfont 
There may be polynomials $p\in\PP_0 $, such that
$$
 \check Q\ast pc_{\xi,\xi}\ast \check Q=0 
 $$
 as functional on $\S(G)\cap \ker\pi $.  
Therefore we always consider only subspaces $W$ of $\PP_0 $ containing the space 
$$
 \PP_Q^0:=\{p\in\PP_0 \mid  \check Q\ast pc_{\xi,\xi}\ast \check Q= 0\text{ on }\S(G)\cap\ker\pi\}.
$$
It follows then by (\ref{vaniv:1}) that for $p\in\PP_Q^0$ we have that
\begin{eqnarray}\label{ppooisdiff}
 {\check Q\ast pc_{\xi,\xi}\ast \check Q}=\langle{pc_{\xi,\xi}},{Q}\rangle c_{\xi,\xi}.
 \end{eqnarray}
 \end{remark}
 
 \begin{proposition}\label{pqideter}
For every $I\in\II^\pi $ and $\xi\in \H_\pi^\iy, Q=Q_\xi $, we have that
\begin{eqnarray*}
 \Pc_{Q}^I=\{p\in\PP_0\vert \check Q\ast pc_{\xi,\xi}\ast\check Q= 0 \; \text{on} \;  I\cap\Sc(G)\}.
 \end{eqnarray*}
In particular, $\Pc_Q^0 =\Pc_Q^{\opn{ker}(\pi)}$. 
 \end{proposition}

\begin{proof} 
By definition we have that 
$\PP_Q^I\subset \{p\in\PP_0\vert \check Q\ast pc_{\xi,\xi}\ast\check Q \; \text{on} \; I\cap\Sc(G) \} $. 
Let now $p\in\PP_0 $, such that $\check Q\ast pc_{\xi,\xi}\ast\check Q\in \S'(G)$ vanishes on $I\cap\Sc(G)$. 
Since $\S(G)\cap I$ is dense in $I $ by Theorem~\ref{sdense}, it follows that the finite dimensional spaces $Q\ast(I/j(\pi))\ast Q $ and $Q\ast (I\cap\S(G))/j(\pi)\cap\S(G) )\ast Q $ are isomorphic. 
Therefore there exists $\psi\in I^\perp $, such that 
$\check Q\ast \ps\ast\check Q=  \check Q \ast \check Q\ast pc_{\xi,\xi}\ast\check Q \ast \check Q= \check Q\ast pc_{\xi,\xi}\ast\check Q $ on $\ker\pi\cap \S(G) $. 
By Proposition \ref{expjpi}, there exist $p_\ps\in \PP_0 $ such that $\check Q\ast \ps\ast\check Q= \check Q\ast p_\ps  c_{\xi,\xi}\ast\check Q$ on $\S(G) $. 
Hence 
\begin{eqnarray*}
 p-p_\ps\in \PP_0^0,
 \end{eqnarray*}
therefore (by \ref{ppooisdiff}) $ \check Q\ast (\ps-\ps_p)c_{\xi,\xi}\ast\check Q=\langle{(p-p_\ps) c_{\xi,\xi}},{Q}\rangle c_{\xi,\xi}$, and thus
\begin{eqnarray*}
 \check Q\ast pc_{\xi,\xi}\ast\check Q=\check Q\ast (\ps+\langle{pc_{\xi,\xi}},{Q}\rangle   c_{\xi,\xi})\ast\check Q.
 \end{eqnarray*}
It follows that  $p\in\PP_Q^I $.
\end{proof}

\begin{definition}\label{invariantdef}
For $x$, $y\in G $ the function $\lambda(x)\rho(y) (\check Q\ast p c_{\xi,\xi}\ast  \check Q )$ is contained in $I^\perp $ whenever $p\in \PP^I_Q$. 
Thus  there exists $_xp_y\in\PP^I_Q $, unique modulo $\PP_Q^0$, such that
 $$ 
 \check Q\ast (_xp_y) c_{\xi,\xi}\ast \check Q=\check Q\ast (\lambda(x)\rho(y)(\check Q\ast pc_{\xi,\xi}\ast \check Q))\ast \check Q
 $$
on $\ker\pi\cap\S(G) $. 
We say that a subspace $W $ of $\PP^\pi_Q $ is \textit{invariant} if for every $p\in W $ the polynomials $_xp_y $ are also in $W $ for every $x$, $y\in G $. 
\end{definition}

 \begin{theorem}\label{deteroIp}
Let $\xi\in \H_\pi^\infty $ and  choose a maximal projection  $Q=Q_\xi $ as before. 
Then
there exists an order reversing bijection between the space $\II^\pi $ and the space $\I^\pi_Q $ of 
invariant subspaces of $\PP^\pi_Q $ containing $\PP_Q^0 $.
\end{theorem}

\begin{proof} 
It suffices to use Theorem~\ref{idetiq}:  
The orthogonal of every ideal $K$ contained in $\ker\pi_Q $ is 
an invariant subspace and every invariant subspace in  $\ker\pi_Q $ 
defines an ideal in $\ker\pi_Q/j(\pi)_Q $.
\end{proof}

\begin{definition}\label{wpidef}
\normalfont   
Consider the space of polynomials 
$$
 W_\pi:=V_\pi\cap \PP_0.
$$
Then $W_\pi $ is  a translation invariant subspace and for every smooth  coefficient $\varphi=c_{\et,\de} $ we have that 
$$
 W_\pi \varphi\subset j(\pi)^\perp.
$$
\end{definition}

\begin{definition}\normalfont
For  any translation invariant subspace $W\subset W_\pi $ let
 $$
 I^W:=\{f\in L^1(G) \mid  \dual{W(\ker\pi^\perp)}{f}=0\}.
$$
Then $j (\pi) \subset I^W$ is a closed twosided ideal of $L^1(G)$ contained in $\II^\pi $. 

For $I\in\II^\pi $ consider 
\begin{eqnarray*}
W_I:=\{p\in W_\pi\mid  pc_{\de, \et}\in I^\perp\text{ for all }\de,\et\in\H_\pi^\iy\}.
 \end{eqnarray*}
 Then $W_I $ is a translation invariant subspace of $W_\pi $. 
 \end{definition}

\begin{theorem}\label{wiseqpqinv}
Let $I\in\II^\pi $. 
The there exists a translation invariant subspace $W \subset W_\pi$ such that $I=I^W $ if and only for some $\xi\in \H_\pi^\iy $ and a maximal $Q=Q_\xi\in \S(G) $ we have that the space of polynomials $\PP_Q^I $ is translation invariant modulo $\PP_Q^0 $, that is, there exists a translation-invariant subspace $W_Q\subset W_\pi  $ 
such that $\PP_Q^I=W_Q+\PP_Q^0 $. 
\end{theorem}

\begin{proof}
Let $W\subset W_\pi $ be a translation invariant subspace and take $\xi\in \H_\pi^\infty$, $Q=Q_\xi $
as before. 
Let $p\in W $.
Then  $pc_{\xi,\xi}\in (I^W)^\perp $ and therefore  also
$$
 \check Q\ast pc_{\xi,\xi}\ast \check Q\in (I^W)^\perp. 
$$
There exists then $p_1\in \PP_Q^{I^W} $ such that
\begin{eqnarray*}
 \check  Q\ast pc_{\xi,\xi}\ast \check Q=\check Q\ast p_1c_{\xi,\xi}\ast \check  Q.
 \end{eqnarray*}
This means that $p-p_1\in \PP_Q^0 $, i.e.,  $p\in\PP_Q^{I^W} $. 
Therefore $W+\PP_Q^0\subset \PP_Q^{I^W} $.

For every $\ps\in (I^W)^\perp $, the function $\check Q\ast \ps\ast \check Q  $ is contained in $\check Q\ast Wc_{\xi,\xi}\ast \check Q $,
since $ (I^W)^\perp$ is the weak$^* $-limit of the span of the functions $p'c_{\eta,\xi} $, with $p'\in W$, $\eta$, $\delta\in \H^\infty_\pi$,  and since $\check Q\ast p'c_{\et,\de}\ast \check Q\in \check Q\ast Wc_{\xi,\xi}\ast \check Q $, due to the fact that  $W $ is translation invariant. 

Let  now $p\in \PP_Q^{I^W} $. 
Then  there exists $\ps\in (I^W)^\perp $ such that $\check Q\ast \ps\ast \check Q=\check Q\ast pc_{\xi,\xi}\ast \check  Q\in \check Q\ast Wc_{\xi,\xi}\ast \check Q$. 
Hence our $p\in\PP_Q^{I^W} $ is contained in $W +\PP_Q^0$. 
Therefore $\PP_Q^{I^W}=W $ modulo $\PP_Q^0 $ is translation invariant modulo $\PP_Q^0 $. 

Assume now that $\PP_Q^I $ is translation invariant modulo $\PP_Q^0 $.
For any $\psi\in I^\perp$, there is $p\in \Pc_Q^I$ such that 
\begin{equation}\label{wiseqpqinv-1}
\check Q\ast pc_{\xi, \xi} \ast \check Q = \check Q\ast \psi \ast \check Q.
\end{equation}
Since $\Pc_{Q}^I = W_Q + P_Q^0$, we can take $p\in W^Q$ in \eqref{wiseqpqinv-1}. 
It follows that $\psi= 0$ on $ Q\ast I^{W_Q}\ast Q$. 
Since $\alpha\ast \psi \ast \beta\in I^\perp $ for every $\alpha, \beta\in L^1(G)$ we actually get that
$\psi = 0$ on $L^1(G) \ast Q \ast I^{W_Q}\ast Q \ast L^1(G)$, hence on $I^{W_Q}$. 
Hence $I^\perp \subseteq (I^{W_Q})^\perp$. 

On the other hand, if $p\in W_Q$, then  all its derivatives  belong also  to $\PP_Q^I $ modulo $\PP_Q^0 $. 
In particular if 
$\{p=p_m,\dots, p_1=1\}$ 
is  a Jordan-H\"older basis for the $G\times G $-module $W_Q\subset \PP_Q^I $ then all the $p_i $'s are contained in $\PP_Q^I $ modulo $\PP_Q^0 $.
Hence for all $\alpha, \beta\in \Sc(G)$  there is $\psi_{\alpha, \beta}\in I^\perp$ such that
$$ \check Q\ast \alpha \ast pc_{\xi, \xi} \ast \beta\ast  \check Q = \check Q\ast \psi_{\alpha, \beta} \ast \check Q. $$ 
Therefore $pc_{\xi, \xi} = 0$ on $\Sc(G)\ast Q\ast I^\perp\ast Q\ast \Sc(G)$, and thus $pc_{\xi, \xi}\in I^\perp$. 
Similarly we have that $p_j c_{\xi,\xi}\in I^\perp$ for every  $j=1, \dots, m-1$. 
It follows that $p \lambda(x)\rho(y)c_{\xi, \xi}\in I^\perp$ for every $x, y\in G$, and since $ p\in W_Q\subseteq W_\pi$, we get that
$pc_{\eta, \delta}\in I^\perp$ for all $\eta, \delta\in \Hc^\infty_\pi$.
Thus we  obtain that 
$(I^{W_Q})^\perp \subseteq I^\perp$. 
This finishes the proof of the theorem. 
\end{proof}

\begin{remark}
 \normalfont
We can choose  the space $\Pc_0\subseteq \Pc(G)$ such that $I^{W_\pi} = K_\pi$.

Indeed, recall that $W_\pi=V_\pi\cap \PP_0 $, $
 K_\pi:=\{f\in L^1(G)\mid \dual{pc_{\xi,\eta}}{f}=0,\,  \forall p\in V_\pi,\, \forall  \xi,\eta \in \Hc_\pi^\infty \}
$ and that $\Pc_0$ in Theorem~\ref{Jpi0cor} was chosen as
$$ \{ p\in \Pc(G) \mid \opn{dim} \Vc_p \le N_0\}$$
for some $N_0$ large enough. 
For $N\in \NN$, $N\ge 0$ denote
$$ \Pc_N := \{ p   \in \Pc(G) \mid \opn{dim} \Vc_p \le N+N_0\}.$$
Then in all the above considerations $\Pc_0$ can be replaced by $\Pc_N$. 

Let now $W_{\pi, N} = V_\pi \cap \PP_N$.
These are $G\times G$ invariant subspaces of $V_\pi$, and they define   
a sequence of closed bilateral ideals 
$$I_N= \{ f\in L^1(G)\mid \dual{W_{\pi, N}\Cc_\pi}{f} =0\}. $$
Then  we have 
$$ K_\pi\subseteq \cdots  \subseteq I_{N} \subseteq I_{N-1} \subseteq \cdots \subseteq I_0=I^{W_\pi}.$$
This corresponds to a sequence of ideals
$$ 
(K_\pi)_Q \subseteq \cdots  \subseteq (I_{N})_Q \subseteq (I_{N-1})_Q \subseteq\cdots \subseteq (I_0)_Q=(I^{W_\pi})_Q
$$ in the finite dimensional algebra 
$(\ker\pi)_Q$. 
Hence there must be an $N_1$ such that $(I_{N})_Q=(I_{N_1})_Q$ for $N\ge N_1$, and by Theorem~\ref{idetiq} it follows that $I_{N} = I_{N_1}$ for $N\ge N_1$. 
On the other hand we have that $K_{\pi} = \bigcap\limits_{N\ge N_0} I_{N}$, therefore we have that
$K_{\pi} = I_{N_1}$. 
We can then replace $\Pc_0$ by $\Pc_{N_1}$, and we take into account that the space $I^{W_\pi}$ for this new
$\Pc_0$ is precisely $I_{N_1}$. 
\end{remark}

\section{$ L^\iy(G/N(\pi))$-invariant ideals}

Let $\pi\colon G \mapsto \BB(\Hc_\pi)$  be an  irreducible  
unitary representation of $G$, as before, 
 and let $\Oc_\pi$ be the corresponding coadjoint orbit.

\begin{definition}\label{nldef}
\normalfont
 a)  Let  $\n=\n(\pi)$ be the subalgebra generated by $\{\g(\ell)\}_{\ell \in\Oc_\pi}$.
Then $ \n(\pi)$ is an ideal since it is a  $G$-invariant subalgebra. 
Furthermore, we know that
\begin{eqnarray*}
 \O_\pi=\O_\pi+\n(\pi)^\perp, 
\end{eqnarray*}
i.e., the coadjoint orbit is saturated with respect to $ \n(\pi)$. 

b) Let 
\begin{eqnarray*}
N= N(\pi)=\exp {\n(\pi)}.
 \end{eqnarray*}
Then $ N(\pi)$ is a closed connected normal subgroup of $ G$

c) Choose a subspace $\x $ in $\g $ such that $\g=\x\oplus\n(\pi) $. 
Let $\X:=\exp\x $. 
The group $G $ is then the topological product of $\X $ and $N(\pi) $. 

d) For a function $p\colon G\to \C $ let $p_\n\colon G\to\C $ be defined by
\begin{eqnarray*}
 p_\n(xn):=p(n), \quad x\in\X,\;  n\in N(\pi).
 \end{eqnarray*}
Then let 
\begin{eqnarray*}
 \PP_{0,\n}:= \{p_n\vert \ p\in\PP_0\}.
 \end{eqnarray*}
 \end{definition}

\begin{definition}\label{iinvariant}
\normalfont
An ideal $ I\in\II^\pi$ is called $ L^\iy(G/N(\pi)) $-invariant  if and only if for all $ f\in I$
\begin{eqnarray*}
 \varphi f\in I\quad  \text{for all} \;  \varphi\in L^\iy(G).
 \end{eqnarray*}
 We denote by $\II^{\pi,N(\pi)} $ the collection of all $I\in\II^\pi $ which are $L^\iy(G/N(\pi)) $-invariant.
 \end{definition}
 
 Since $ \O_\pi$ is saturated with respect to $ \n(\pi)$ we know that the ideals $ \ker\pi^j$, $j=1,2,\dots $,  and $ j(\pi)$ are $ L^\iy(G/N(\pi))$ invariant. 
 Indeed we have by \cite{Ki62} that
\begin{eqnarray*}
 \ker\pi=\ker{\ind_{N(\pi)}^G {\pi\res{N(\pi)}}}.
 \end{eqnarray*}

\begin{lemma}\label{scoeff}
Let $p= p(s)$ be a polynomial that depends on $\Xc$ only. 
Then $ p\Phi\in \ker{\pi}^\perp \cap BC^\infty(G)$ for every  $\Phi\in \ker{\pi}^\perp \cap BC^\infty(G)$.
\end{lemma}

\begin{proof}
It is enough to show that $pc^\pi_{\pi(\check{\varphi})}\in \ker{\pi}^\perp$ 
for all $\varphi\in \Sc(G)$. 

By \cite[Lemma~2.3]{Pu79} it follows that there is a $G$-invariant measure $d\ell'$ on $\Oc_\pi\vert_\n$ such that
$$\int\limits_{\Oc_\pi} \hat{\varphi}(\xi) \dd \ell = \int \limits_{\Oc_\pi\vert_\n} \widehat{\varphi_\n}(\ell')\dd \ell'.$$ 
Thus we have
$$
c^\pi_{\pi(\varphi)} (g)  =\tr(\pi(g^{-1})\pi (\varphi))
 = C_\pi \int\limits_{\Oc_\pi \vert_\n} \widehat{(\lambda(g^{-1}) \varphi)\res\n}(\ell')\dd\ell', 
$$
with $C_\pi$ a constant.  
On the other hand we have $ (\lambda(g^{-1}) \varphi) \res\n(Y)= \varphi( g\exp Y )$ for $Y\in \n$, thus
$$ (\widehat{\lambda(g^{-1}) \varphi)\res\n}(\ell')=\int\limits_{\n} e^{-i\dual{\ell'}{Y}}\varphi(g\exp{Y}) \dd Y, \quad \ell'\in \Oc_\pi\vert_\n.
$$
Hence, if $p(sn)=p(s)$ for every $n\in N_\pi$, $s\in \Xc$,  we have that 
$$ p(sn) \widehat{(\lambda(sn)\varphi)\res\n}(\ell') = \widehat{(\lambda(sn) \psi)\res\n}(\ell'), $$
where $\psi = p\varphi$. 
It follows that $p c^\pi_{\pi(\varphi)}=c^\pi_{\pi(\psi)}$, and this finishes the proof of the lemma.  
\end{proof}

 \begin{lemma}\label{resiniperp}
Let $I\in\II^{\pi,N(\pi)} $ and $\ps$ in $I^\perp \cap C(G) $. 
Then the function 
$\psi_n $
 is also contained in $I^\perp $. 
 \end{lemma}
\begin{proof}
Indeed, we have for $f\in I\cap \S(G)$, $\varphi\in \L\iy{G/N(\pi)}\cap C(G) $,  and  $t\in G$, 
\begin{eqnarray*}
 0=\int\limits_G \ps(g)\va(g)f(tg)dg=\int\limits_{G/N(\pi)}\va(s) \int\limits_{\np} f(tsn)\ps(sn)\, dn\, d\dot s.
 \end{eqnarray*}
Then 
\begin{eqnarray*}
 0= \int\limits_{\np} f(tn)\ps(n)dn, \quad \text{for all}\;  t\in G.
 \end{eqnarray*}
Hence
\begin{eqnarray*}
 0= \int\limits_{G/\np}\int\limits_{\np} f(tn)\ps(sn)dnd\dot t=\int\limits _{G}f(g)\ps_\n(g)dg.
 \end{eqnarray*}
Therefore $\ps_\n\in I^\perp  $.
\end{proof}

\begin{lemma}\label{wisinv}
Let $ W\subset W_\pi$ be a translation invariant subspace. 
Then the ideal $ I^W$ is $ \L\iy {G/N(\pi)}$-invariant.
 \end{lemma}
\begin{proof}
It follows from \cite{Pe94} (see also the proof of the above Lemma~\ref{scoeff}) that for any $ 
\varphi\in\S(G/N(\pi))$ and any smooth coefficient $ c_A$ of $ \pi$ we have that the function  
$\varphi c_A$ is also a smooth coefficient function of $ \pi$. 
Let now $ p\in W$ and $ f\in I^W$. 
Then 
$$
\dual{pc_A}{f} =0 \quad    \text{for all}\;  A\in \BB(\Hc_\pi)_\infty.
$$
Hence for any $ \varphi\in \S(G/N(\pi))$
\begin{eqnarray*}
 \dual{pc_A}{\va f}=\dual{p(\va c_A)}{ f}\rangle =0\quad   \text{for all}\;  A\in \BB(\Hc_\pi)_\infty.
 \end{eqnarray*}
Thus we get that $ \va f\in I^W$.
\end{proof}

\begin{proposition}\label{iothded}
Let $I\in \II^{\pi,\np} $. Then
\begin{eqnarray*}
 I^\perp \cap \PP_{0,\n}\Cc_\pi
 \end{eqnarray*}
is weak$^*$ dense in $I^\perp  $.
 \end{proposition}
\begin{proof} 
We know from Proposition \ref{mqgene} that the span of the subspaces $\check Q_\de\ast I^\perp \ast\check Q_\eta$,
with $ \delta,\eta\in \H_\pi^\iy $ is weak$^*$ dense in $I^\perp  $. 
Now we can write for $\psi\in I^\perp $
\begin{eqnarray*}
 \check Q_\delta\ast \psi \ast\check Q_\eta=pc_{\eta,\delta}+\sum_{j=1}^{m-1}p_j c_{A_j}
 \end{eqnarray*}
for some $p\in\PP_0  $ and some smooth operators $A_j$, $j=1,\dots, m-1$, where 
$\{p=p_m,\dots, p_1\} $ is  a Jordan-H\"older basis of $\V_p $.
On the other hand for $x\in\X$, $n\in\n_\pi $, we have that
$$
\begin{aligned}
 (\check Q_\delta\ast \ps \ast\check Q_\eta)(xn)&=p(xn)c_{\eta,\delta}(xn)+\sum_{j=1}^{m-1}p_j (xn)c_{A_j}(xn)\\
 &=(p(n)+\sum_{j=1}^{m-1}\check{a}_{j,m-1}(x)p_j(n))c_{\eta,\delta}(xn)\\
 &\qquad + \sum_{j=1}^{m-1}(p_j(n)+\sum_{i=1}^{j-1}\check {a}_{i,j}(x)p_i(n))c_{A_j}(xn) \\
 &=p_\n(xn) c_{\eta,\delta}(xn) +\sum_{j=1}^{m-1}(p_j)_\n(xn)c_{B_j}(xn)
 \end{aligned}
$$
for some $B_j\in \mathbb B(\H_\pi)_\iy $, where we have used Lemma~\ref{scoeff}.
Hence $\check Q_\delta\ast \ps \ast\check Q_\eta \in \PP_{0,\n}\Cc_\pi$, which finishes the proof.
\end{proof}

\section{Examples}\label{examples}
   
\subsubsection*{1. Flat coadjoint orbits}
Assume that $G$ is a connected simply connected
 Lie group with Lie algebra $\gg$, 
and $\ell\in \gg^\ast$ is such that the corresponding orbit is flat, that is, 
$\Oc_\ell= \ell + \gg(\ell)^\perp$, or, equivalently, $\gg(\ell)$ is an ideal in $\gg$.

It is well-known that $\Oc_\ell$ corresponds to a set $\{[\pi_\ell]\}$ of spectral synthesis in $\hat G$ (see \cite{HL81}). 
Let us present below another way of seeing this.

Let $\gg_0$ be a complement of $\gg(\ell)$, 
$$ \gg = \gg_0 + \gg_\ell, $$
and identify $G$ and $\gg$ via the exponential mapping. 
With this notation a polynomial $p\in \Pc(G)$ can be written as
$$ p(x + s) = \sum\limits x^\alpha p_\alpha(s), \quad x\in \gg_0, \, s\in \gg(\ell).$$
Here $x^\alpha$ corresponds to the polynomial $t_1^{\alpha_1} \cdots t_k^{\alpha_k}$,  $k =\opn{dim}\gg_0$, 
in the coordinates of the first kind, 
 and are linearly independent polynomials that generate $\Pc(\gg_0)$.
Since $\Oc_\ell$ is flat, for $p \in \Pc(G)$, $p(x+s) =\sum\limits x^\alpha p_\alpha(s)$ and $f\in \Sc(G)$,  we have that 
$$ pf \in \ker{\pi_\ell} \Leftrightarrow \int\limits_{\Oc_\ell} \widehat{pf} (\xi) d \xi = 0.$$
(See \cite[Thm.~1]{Lu86}.) 
Here $\hat f$ denotes the Fourier transform on $\gg$, and $d\xi$ is the invariant normalized Liouville measure on $\Oc_\ell$. 
With the identification $\Oc= \ell + \gg_0^\ast$ this is further equivalent with the fact 
that 
$$ \sum\limits q_\alpha (x) \Fc_{\gg(\ell)}(p_\alpha f)(x, \ell) =  0\quad  \text{for all } x\in \gg_0,  $$
where $\Fc_{\gg(\ell)}$ is the partial Fourier transform in variable $s \in \gg(\ell)$. 
We get thus that, in this case, the space $K_{\pi_\ell, 0}$ defined as in Lemma~\ref{Kpi0} consists of all 
$f\in \Sc(G)$ such that
$$ \mathbf{D}_{\gg(\ell)} \Fc_{\gg(\ell)}((p_\alpha f)(x, \ell) =0, \quad x\in\gg_0, $$
for every $\mathbf{D}_{\gg(\ell)} $ differential operator on $\gg(\ell)$. 
The closure of this space in $L^1(G)$ is nothing else than 
$$ \{ f \in L^1(G) \mid \Fc_{\gg(\ell)}(p_\alpha f)(x, \ell) =0, \; \forall x\in \gg_0\} =\ker \pi.$$
Thus we find again that $\{\pi_\ell\}$ is of spectral synthesis.

Note that for the corresponding representation $\pi=\pi_\ell$ we have that
$$ \vert c_{\xi, \eta}(x +s )\vert   = \vert c_{\xi, \eta}(x) \vert \, \quad \text{for all} \;  s\in \gg(\ell), x\in \gg_0, $$ 
and 
$$ c_{\xi, \eta} \in \Sc (\gg_0).$$
It follows that $V_\pi$ consists of polynomials $p$ such that $p(x+s) = p(x)$, for all $x\in \gg_0$, $s\in \gg(\ell) $. 
On the other hand, by using Lemma~\ref{scoeff} for $p=p(x)$ and $\xi, \eta\in \Hc_{\pi_\ell}$   there is an $f_p\in \Sc(G)$ such that 
$$ p c_{\xi, \eta}  = c_{\pi_\ell(f_p)}.$$
We immediately see therefore that $V_{\pi_\ell} \Cc_{\pi_\ell} \subseteq \ker{\pi_\ell}^\perp$.

\subsubsection*{2. Step $3$ nilpotent Lie groups}
The complete description of the structure of primary ideals of nilpotent Lie groups of step $3$ can be found in \cite{Lu83a}, and we refer to this paper for details of the computations below. 
Here we briefly show how our present results can be used to find the ideals in $\II^\pi$ that are of the form 
$I^W$ with $W$ translation invariant subspace in $V_\pi$.

Let $G$ be a  step $3$ nilpotent Lie group, connected and simply connected, and let $\gg$ be its Lie algebra. 
We can assume  that 
$$ [\gg, [\gg, \gg]]= \zg(\gg) = \RR z $$
the centre $\zg(\gg)$ is one dimensional,  
$\ell \in \gg^\ast$ satisfies $\dual{l}{\zg(\gg)}\ne 0$, and   $0\ne z\in \zg$ is chosen such that
$\dual{l}{z}=1$.
Let also $y_1, \dots, y_k$ be a basis  of $[\gg, \gg]$.
Then there exist $x_1, \dots, x_k\in \gg$ such that
$$ [x_j, y_l]=\delta_{jl} z,  \quad j, l=1, \dots, k. $$ 
Consider now the two-step subalgebra
$$ \hg= \{ h  \in \gg \mid [h, [\gg, \gg]]=0\}, $$
which is also an ideal in $\gg$. 
Then the centre of $\hg$
is the abelian subalgebra $\gg_0= \gg(l) +[\gg, \gg]$.  
Note that $\gg_0$ is an ideal in $\g$ and contains $\gg(\ell)$. 

The representation $\pi$ associated with $\ell$ can be realized as
$$ \pi =\opn{ind}^G_H \pi_0$$
where $\pi_0$ is the representation of $H$ associated with $\ell\big \vert_\hg$. 

If we denote $\Xc=\opn{span}\{x_1, \dots, x_k\}$ and $\Xit = \opn{exp}\Xc$, then 
we can write $G =  \Xit\cdot  H$. 

Below we identify the groups with their Lie algebras, via the exponential mapping. 
Let now $f\in L^2(\Xit, \Hc_{\pi_0})$, 
and for  $x\in \Xit$,  $h \in H$ define 
$\tilde{f}(xh) = \pi_0(h^{-1})( f(x))$. 
Then we have for $u \in \Xit$ 
$$ \begin{aligned}
\pi(x h)f(u) = & \tilde{f}(h^{-1} x^{-1} u) = \tilde{f}(x^{-1} u u^{-1} x h^{-1} x^{-1} u)\\
= & \pi_0 (u^{-1} x h  x^{-1} u)(\tilde{f}(x^{-1} u))\\
= & e^{i\dual{\ell}{[u^{-1} x, h]+ \frac{1}{2} [u^{-1} x, [u^{-1} x, h]]}}
\pi_0 (h) (\tilde{f}(x^{-1} u)).
\end{aligned}
$$
Note that $[u^{-1}x, h]=  [-u+x, h]$ for every $h$ in $\hg$, and 
$$x^{-1} u = (u-x) (-\frac{1}{2}[x, u] -\frac{1}{6}[x, [x, u]]- \frac{1}{3}[u, [u, x]]), $$
and $(-\frac{1}{2}[x, u] -\frac{1}{6}[x, [x, u]]- \frac{1}{3}[u, [u, x]])\in  \gg_0$. 

Take now $f= \phi \otimes\xi$, $g = \psi \otimes \eta$,  where $\phi, \psi \in \Sc(\Xc)$ and $\xi, \eta \in \Hc^\infty_{\pi_0}$. 
Also write $h \in \hg $ as $h=  t ns$ where $n \in [\gg, \gg]$, and $s\in \gg(l)/[\gg, \gg]$.
Then 
$$
\begin{aligned} 
c^\pi_{f, g} (xt ns)  &  =  \dual{f}{\pi_\ell (xtns)g}\\
  & = c^{\pi_0}_{\xi, \eta}(t)   \int\limits_{\Xc}  e^{-i\chi(t, s, n, x, u) }\phi(u)\overline{\psi(u-x)}  \, du
 \end{aligned}
 $$
 where
  $$ \begin{aligned}
  \chi(t, s, n, x, u) & = \dual{\ell}{ s+n}+ \dual{\ell}{[u-x, t+n]}+ \dual{\ell}{[u-x, [u-x, t+ s+n]]/2} \\ 
  &  -\dual{\ell}{[x, u]/2 +[x, [x, u]]/6+[u, [u, x]/3]}.
  \end{aligned}
 $$
 Here
$$ 
\begin{aligned}
 \;  [u-x, [u-x, s+t]] &  =  [u-x, [u-x, s]]\\
 & =  [u, [u, s]]+ [x, [x, s]] - [x, [u, s]] -[u, [x, s]]\\
                & = [u, [u, s]]+ [x, [x, s]] - 2[u, [x, s]] + [s, [u, x]].
\end{aligned}
$$ 
Since $s\in \gg(l)$ we get 
\begin{equation}
\label{c:1}
\begin{aligned}
c^\pi_{f, g} (xtns) = e^{-i\dual{\ell}{s+n+t}}  e^{-i \dual{\ell}{[x,[x, s]]-[x, n]}/2}  
c^{\pi_0}_{\xi, \eta}(t)
\\
\times \int\limits_{\Xc} e^{-i\dual{\ell}{[u, n+t]}} e^{i\dual{\ell}{[u, [x, s]]}} 
e^{-i\dual{\ell}{[u, [u, s]]/2}} \phi(u) \overline{\psi_1(u, x)} \, du. 
\end{aligned}
\end{equation}
where $\psi_1(u, x) = e^{-i \dual{\ell}{[x, u]/2 +[x, [x, u]]/6+[u, [u, x]/3]}}\psi(u-t)$.

Recall that $\dual{\ell}{[u, n]}= \dual{u}{n}$, the bracket in the right-hand side being  the euclidian scalar product. 
We get by \eqref{c:1} that 
\begin{equation}
\label{c:2}
\begin{aligned}
c^\pi_{f, g} (xtns)  & = (2\pi)^{-k}  e^{-i\dual{\ell}{s+n+t }} 
 e^{-i \dual{\ell}{[x,[x, s]]-[x, n]}/2} c^{\pi_0}_{\xi, \eta}(t)\times \\
 & \times \int\limits_{\Xc^\ast}\hat{G}(x, n-\ell\circ \opn{ad}([x, s])-\ell\circ \opn{ad}(t)-\xi) \hat{\Phi} (s, \xi)  \, d\xi,  
\end{aligned}
\end{equation}
where, for $N$ large enough (to be chosen later on), 
\begin{equation}
 \label{c:3}
\begin{aligned}
 G(x, u) & = \overline{\psi_1(u, x)}(1+\vert u\vert)^{-N}\\ 
\Phi(s, u) & =  e^{-i \dual{\ell}{[u, [u, s]]/2}} \phi(u) (1+\vert u\vert)^{N} 
\end{aligned}
\end{equation}
and the Fourier transform is considered in the second variable of these functions. 

A computation similar to the one in \cite{Lu83a} shows that 
\begin{equation}\label{c:4}
 \vert \hat\Phi(s, \xi)\vert \le C (\det (A(s)^2 +1))^{-1/4}, 
\end{equation}
where $A(s)$ is the bilinear form
$$ \dual{A(s)u}{u} =\dual{\ell}{[u, [u, s]]}.
$$

It remains to estimate the $L^1$ norm of $\hat{G}(x, \xi)$. 
Note that
$$
G(x, u)   =  e^{i \dual{\ell}{[x, u]/2 +[x, [x, u]]/6+[u, [u, x]/3]}}\overline{\psi(u-x)}(1+\vert u\vert)^{-N}
$$
 Since  $\psi\in \Sc(\Xit)$ we get that we can chose $N$ large enough such that 
there is $C$ with 
$$ \vert \partial_u^\alpha G(x, u) \vert \le C_{\alpha, N}(1+|u|)^{-k-1} .$$
Hence there is $C>0$ such that for all $x\in \Xc$
\begin{equation}\label{c:5}
\norm{\hat G(x, \cdot)}_{L^1(\Xc^\ast)} \le C.  
\end{equation}
Summing up \eqref{c:5} and \eqref{c:4} in 
\eqref{c:3} we get that 
$$ 
\sup_{x\in \Xc}\vert c^\pi_{f, g}(xtsn)\vert\le C \vert c_{\xi, \eta}^{\pi_0}(t) \vert (\det(A(s)^2+1))^{-1/4}. 
$$

Let $\omega$ be the weight on $\gg_0$ defined by  $\omega(s) =  (\det(A(s)^2+1))^{1/4}$. 
Then we have obtained that if a polynomial $p$ is such that $\vert p(xtns)\vert\le C \omega(s)^{-1}$, together with its $G$-derivatives, then 
it belongs to $V_\pi$.
 Combining this with Lemma~\ref{scoeff}, we see that the ideals in $\II^\pi$ of the form
$I^W$ with $W\in V_\pi$ correspond to linear, translation invariant subspaces of polynomials on $\gg_0$ that are bounded by $\omega$. 
By the result in \cite{Lu83a} we actually know that these are all the primary ideals in $\II^\pi$.

\end{document}